\documentclass[a4paper,12pt,reqno]{amsart}
\usepackage{lmodern,amsmath,amssymb,mathtools,amsthm}
\usepackage{enumerate,cite}
\usepackage{graphics,graphicx,geometry}
\usepackage{hyperref}
\usepackage{mathrsfs}
\usepackage{wasysym}
\usepackage{cleveref}
\usepackage{floatrow,morefloats}
\usepackage{epsfig}
\usepackage{epstopdf}
\usepackage{xcolor}
\usepackage{caption,subcaption}
\usepackage{exscale, relsize}
\usepackage{hyphenat}

\numberwithin{equation}{section}

\newtheorem{theorem}{Theorem}
\newtheorem{proposition}{Proposition}

\newtheorem{remark}{Remark}
\newtheorem{example}{Example}

\numberwithin{definition}{section}
\numberwithin{lemma}{section}
\numberwithin{corollary}{section}
\numberwithin{remark}{section}
\numberwithin{example}{section}

\title{Orthogonality of quasi-nature spectral polynomials of Jacobi and Laguerre type}

\author{VIKASH KUMAR$^\dagger$}
\address{$^\dagger$Department of Mathematics\\ Indian Institute of Technology, Roorkee-247667, Uttarakhand, India}
\email{vikaskr0006@gmail.com, vkumar4@mt.iitr.ac.in}

\author{A. Swaminathan$^\ddagger$}
\address{$^\ddagger$Department of Mathematics\\ Indian Institute of Technology, Roorkee-247667, Uttarakhand, India}
\email{mathswami@gmail.com, a.swaminathan@ma.iitr.ac.in}
\allowdisplaybreaks
\bigskip
\begin{document}

	\begin{abstract}
	In this work, the explicit expressions of coefficients involved in quasi Christoffel polynomials of order one and quasi-Geronimus polynomials of order one are determined for Jacobi polynomials. These coefficients are responsible for establishing the orthogonality of quasi-spectral polynomials of Jacobi polynomials.  Additionally, the orthogonality of quasi-Christoffel Laguerre polynomials of order one is derived. In the process of achieving orthogonality, in both cases, one zero is located on the boundary of the support of the measure. This allows us to derive the chain sequence and minimal parameter sequence at the point lying at the end point of the support of the measure.
 Furthermore, the interlacing properties among the zeros of quasi-spectral orthogonal Jacobi polynomials and Jacobi polynomials are illustrated. Finally, we define the quasi-Christoffel polynomials of order one on the unit circle and analyze the location of their zeros for specific examples, as well as propose the problem in the general setup.
	\end{abstract}
	\maketitle
	\markboth{Vikash Kumar and A. Swaminathan}{QUASI-SPECTRAL TRANSFORMATIONS}

\begin{center}
Dedicated to the memory of Professor Lawrence A. Zalcman
\end{center}

\section{Introduction}
Suppose $\{\mathbb{P}_n(x)\}$ denotes the sequence of monic orthogonal polynomial with respect to the linear functional $\mathcal{L}$. The orthogonal polynomial $\mathbb{P}_n(x)$ for $n\geq1$ can be generated by the three term recurrence relation:
\begin{align}\label{TTRR}
	x\mathbb{P}_n(x)=\mathbb{P}_{n+1}(x)+c_{n+1}\mathbb{P}_{n}(x)+\lambda_{n+1}\mathbb{P}_{n-1}(x),
\end{align}
with initial conditions $\mathbb{P}_{-1}(x)=0, \mathbb{P}_0(x)=1$, where the recurrence parameters $c_n$ and $\lambda_{n}$ are given by:
\begin{align}
	c_{n+1}=\frac{\mathcal{L}[x\mathbb{P}^2_n(x)]}{\mathcal{L}[\mathbb{P}^2_n(x)]},~~\lambda_{n+1}=\frac{\mathcal{L}[\mathbb{P}^2_n(x)]}{\mathcal{L}[\mathbb{P}^2_{n-1}(x)]}.
\end{align}
The positivity of $\lambda_{n}$'s depend upon the definitness property of the linear functional $\mathcal{L}$ (see \cite{Chihara book}). Using \eqref{TTRR}, we obtain the tridiagonal matrix, known as the Jacobi matrix for the monic orthogonal polynomials, as:
\begin{align}
	\mathcal{J}=\left(\begin{array}{cccccc}
		c_1 & 1 & 0 & 0&0&\cdots \\
		\lambda_2 & c_2 & 1 & 0&0 &\cdot \\
		0 &  \lambda_3 & c_3& 1&0&\cdot \\
		0 & 0 & \lambda_4 & c_4 & 1&\cdot\\
		\vdots & \vdots & ~ & \ddots & \ddots&\ddots
	\end{array}\right).
\end{align}

The polynomial of degree $n$ is termed quasi-orthogonal of order $k$ with respect to a linear functional $\mathcal{L}$ over the interval $(c,d)$ if it satisfies the condition:
\begin{align}
	\mathcal{L}(x^rq(x))=0~~\text{for}~~r=0,1,...,n-k-1.
\end{align}
A necessary and sufficient condition for a monic polynomial $q(x)$ of degree $n$ to be a quasi-orthogonal polynomial of order $k$ is that the polynomial $q(x)$ can be expressed as a linear combination of orthogonal polynomials $\{\mathbb{P}_j(x)\}_{j=n-k}^{n}$ with constant coefficients, i.e., \begin{align}\label{quasi orthogonal for general order}
	q(x)=\mathbb{P}_n(x)+b^{(n)}_{n}\mathbb{P}_{n-1}(x)+b^{(n)}_{n-1}\mathbb{P}_{n-2}(x)+...+b^{(n)}_{n-k+1}\mathbb{P}_{n-k}(x),
\end{align}
where coefficients $b^{(n)}_j$'s cannot all be zero simultaneously.
The history of quasi-orthogonal polynomials dates back to the work of Riesz in 1923. In  \cite{Riesz}, Riesz  introduced quasi\hyp{}orthogonal polynomials of order one in the proof of the Hamburger moment problem. Subsequently, Shohat extended this concept to finite order in \cite{Shohat} and applied it in the study of mechanical quadrature formulas, known as the Riesz\hyp{}Shohat theorem \cite{Peherstorfer_quasi orthogonal_MOC_1996,Sawa_Uchida_CD kernel_COP_Tran. AMS_2020}. Particularly, the case $k=1$ is of special interest. The necessary and sufficient conditions on the coefficients of quasi\hyp{}orthogonal polynomials of order one are imposed in \cite{Ismail_2019_quasi-orthogonal} to achieve orthogonality. They also derive the relation between the Jacobi matrix for orthogonal polynomials and the Jacobi matrix for quasi-orthogonal polynomials after achieving orthogonality.
It is shown in \cite{Shohat} that at most $k$ zeros of the polynomial $q(x)$ defined in \eqref{quasi orthogonal for general order} lie outside the support of the measure. In \cite{Beardon_Driver_Interlace quasi OP and OP}, interlacing properties among the zeros of polynomial $q(x)$ of degree $m\leq n$ and polynomial $\mathbb{P}_n(x)$ of degree $n$ are discussed.  The properties of quasi-orthogonality of discrete orthogonal polynomials,  the Meixner polynomial, including the zeros and the interlacing property of zeros, are discussed in \cite{Jooste_Jordaan_zeros_quasi Meixner_Dolomites_2023}. For additional insights into quasi-orthogonal polynomials or self-perturbed polynomials, we encourage interested readers to explore \cite{Chihara book, When do linear, Draux2016, Duran_D-operator_Krall_polynomials_JAT_2013} and the references therein.

In this work, we consider the quasi-Christoffel and quasi-Geronimus polynomial of order one of the author's previous work \cite{Vikas_Paco_Swami_quasi-nature,Vikas_Swami_quasi-type kernel}.
This manuscript serves a dual purpose. Firstly, it focuses on deriving the explicit expressions of coefficients essential for achieving the orthogonality of quasi-Christoffel and quasi-Geronimus Jacobi polynomials. The orthogonality of quasi-Christoffel Laguerre polynomial of order one is also achieved. Secondly, it addresses the problem of identifying orthogonal polynomial systems by self-perturbing the Jacobi polynomials so that only one zero lies at one of the end point of the support of the measure.


%
%
%
%

The manuscript is organized as follows: In Section \ref{sec:Quasi-type kernel}, we discuss the orthogonality of quasi-Christoffel Laguerre polynomials of order one and quasi-Christoffel Jacobi polynomials of order one. The relation between the Jacobi matrix for the kernel polynomials and the quasi-Christoffel orthogonal polynomials is discussed. During the process of achieving orthogonality, we observe that one zero of the quasi-Christoffel Jacobi polynomial of order one and the quasi-Christoffel Laguerre polynomial of order one lies at a finite end point of the support of the measure. We present a graphical interpretation illustrating the interlacing of zeros among the Jacobi polynomials, Christoffel transformed Jacobi polynomials, and quasi-Christoffel Jacobi polynomials of order one. Analogous results hold good for Laguerre case as well.  The chain sequence and minimal parameter sequence are also computed at the point lying on the boundary of the support of the measure for  Laguerre polynomials.  In Section \ref{sec:Quasi-Geronimus}, we explore the orthogonality of quasi-Geronimus Jacobi polynomials of order one and provide a graphical representation of the interlacing property between the zeros of quasi-Geronimus and quasi-Christoffel Jacobi polynomial of order one. In Section \ref{Sec:Conclusion}, a numerical illustration is presented to understand the location of the zeros of quasi-Christoffel polynomials of order one on the unit circle involving specific examples.

\section{quasi-Christoffel polynomial of order one}\label{sec:Quasi-type kernel}

Define a linear functional at $a\in \mathbb{C}$ as:
 \begin{align*}
	\mathcal{L}^C(p(x))=\mathcal{L}((x-a)p(x)),
\end{align*}
for any polynomial $p$. The linear functional $\mathcal{L}^C$ is a perturbation of the linear functional $\mathcal{L}$, and this perturbed linear functional is known as the canonical Christoffel transformation at point $a$. The existence of the orthogonal polynomials corresponding to $\mathcal{L}^C$ is guaranteed if $a$ does not annihilate $ \mathbb{P}_n(x)$ for any $n$. The positive definiteness of $\mathcal{L}$, together with the condition that $a$ lies to the left of the interval of orthogonality, ensures the positive definiteness of $\mathcal{L}^C$. Thus, the polynomial corresponding to the linear functional $\mathcal{L}^C$ is called the kernel polynomial or Christoffel polynomial, and is given by:
\begin{equation}\label{kernel}
	\mathcal{C}_n(x;a)=\frac{1}{x-a}\left[\mathbb{P}_{n+1}(x)-\frac{\mathbb{P}_{n+1}(a)}{\mathbb{P}_n(a)} \mathbb{P}_n(x)\right], ~ n\geq 0.
\end{equation}
The polynomial $\mathcal{C}_n(x;a)$ satisfies the TTRR
\begin{align}\label{Kernel-TTRR}x\mathcal{C}_n(x;a)=\mathcal{C}_{n+1}(x;a)+c_{n+1}^c\mathcal{C}_n(x;a)+ \lambda_{n+1}^c\mathcal{C}_{n-1}(x;a), ~ n\geq 0,
\end{align}
where
\begin{align}\label{Kernel recurrence parameters}
	\lambda_{n}^c=\lambda_{n}\frac{\mathbb{P}_{n}(a)\mathbb{P}_{n-2}(a)}{\mathbb{P}_{n-1}^2(a)}, \quad c_n^c=c_{n+1}-\frac{\mathbb{P}_{n}^2(a)-\mathbb{P}_{n-1}(a)\mathbb{P}_{n+1}(a)}{\mathbb{P}_{n-1}(a)\mathbb{P}_{n}(a)}, ~n\geq 1.
\end{align}
This transformation, along with the other two (Geronimus and Uvarov), is useful in studying various generalizations of orthogonal polynomials, such as multiple orthogonal polynomials and Sobolev orthogonal polynomials. For more information on the use of linear spectral transformations, we refer the interested reader to \cite{Manas_SOP_Gaus-Borel factor_2024_AMP,Branquinho_Ana_Manas_MOP_Pearson equation_Christofell_2022_AMP}.

From TTRR given in \eqref{Kernel-TTRR}, we can write an eigenvalue equation whose eigenvectors are the sequence of Christoffel polynomials as
\begin{align}\label{EVequation_Kernelpolynomials}
\left(\begin{array}{ccccc}
	c_1^{c} & 1 & 0 & \cdots \\
	\lambda_2^{c} & c_2^{c} & 1 & \\
	0 & \lambda_3^{c} & c_3^{c} & 1 \\
	0 & 0 & \lambda_4^{c} & c_4^{c} & 1\\
	\vdots & \vdots & \ddots & \ddots & \ddots
\end{array}\right)\left(\begin{array}{c}
	\mathcal{C}_{0}(x;a)\\
	\mathcal{C}_{1}(x;a)\\
	\mathcal{C}_{2}(x;a)\\
	\mathcal{C}_{3}(x;a)\\
	\vdots
\end{array}\right)=x	\left(\begin{array}{c}
		\mathcal{C}_{0}(x;a)\\
		\mathcal{C}_{1}(x;a)\\
		\mathcal{C}_{2}(x;a)\\
		\mathcal{C}_{3}(x;a)\\
		\vdots
	\end{array}\right).
\end{align}
The linear combination of two consecutive degrees of kernel polynomials, known as quasi-type kernel polynomials, is discussed in \cite{Vikas_Swami_quasi-type kernel}. The next result provides the characterization of quasi-type kernel polynomials of order one.
\begin{theorem}{\rm\cite{Vikas_Swami_quasi-type kernel}}\label{quasi-type kernel}
	Let $\mathcal{C}_n(x;a)$ be the monic polynomial with respect to canonical Christoffel transformation, which exists for some point $a$. The monic polynomial	$\mathcal{C}_{n}^Q(x;a)$ of degree $n$ is a non trivial quasi-type kernel polynomial of order one with respect to the linear functional $\mathcal{L}^C$ if and only if there exists a sequence of constants $\gamma_n\not\equiv0$,  such that
	\begin{align}\label{Quasi-Christoffel order one}
		\mathcal{C}_{n+1}^Q(x;a)= \mathcal{C}_{n+1}(x;a) + \gamma_{n+1}\mathcal{C}_{n}(x;a).
	\end{align}
\end{theorem}
\begin{remark}
	It is noted that $(x-a)\mathcal{C}_{n+1}^Q(x;a)$ can be recognized as special case of quasi-orthogonal polynomial of order two with respect to the linear functional $\mathcal{L}$. Hence \eqref{Quasi-Christoffel order one} can be written as:
	\begin{align}\label{quasi-OP special order two}
	(x-a)\mathcal{C}_{n}^Q(x;a)=\mathbb{P}_{n+1}(x)+d_n\mathbb{P}_{n}(x)+e_n\mathbb{P}_{n-1}(x),
	\end{align}
	where $d_n=\left(\gamma_{n}-\frac{\mathbb{P}_{n+1}(a)}{\mathbb{P}_{n}(a)}\right)$ and $e_n=-\gamma_n\frac{\mathbb{P}_{n}(a)}{\mathbb{P}_{n-1}(a)}$.
\end{remark}

Note that the orthogonality of linear combinations of classical discrete orthogonal polynomials, such as Charlier, Meixner, Krawtchouk, and Hahn, is examined in \cite{Duran_D-operator_Krall_polynomials_JAT_2013}. Explicit expressions for the coefficient involved in the linear combination of Laguerre polynomials using the concept of the $\mathcal{D}$-operator associated with orthogonal polynomials to achieve orthogonality is derived in \cite{Duran_D-operator_Krall_polynomials_JAT_2013}. The resulting orthogonal polynomials are known as Krall-Laguerre-Koornwinder polynomials. Similarly, the polynomials obtained from the linear combination of two consecutive Jacobi polynomials are referred to as Krall-Jacobi-Koornwinder polynomials, also discussed in \cite{Duran_D-operator_Krall_polynomials_JAT_2013}.

 Proposition \ref{orthogonality of quasi-type kernel} examines the orthogonality of quasi-type kernel (or quasi-Christoffel) polynomials of order one using the TTRR, imposing certain assumptions on the coefficients $\gamma_n$. Throughout this manuscript, we use quasi-Christoffel polynomial instead of quasi-type kernel polynomial. The following result is part of \cite[Proposition 3]{Vikas_Swami_quasi-type kernel}.
\begin{proposition}\label{orthogonality of quasi-type kernel}
	Let $\mathcal{C}_{n}^{Q}(x;a)$ be a monic quasi-Christoffel polynomial of order one with parameter $\gamma_{n}$ such that
	\begin{align}\label{gamma n restriction cond}
		\gamma_{n}(c_{n+1}^c-c_n^c+\gamma_{n}-\gamma_{n+1})+\frac{\gamma_{n}}{\gamma_{n-1}}\lambda_{n}^c-\lambda_{n+1}^c=0, ~n\geq2.
	\end{align}
	Then the polynomials $\mathcal{C}_{n}^{Q}(x;a)$ satisfy the three-term recurrence relation
	\begin{align}\label{TTRR_QCP}
		\mathcal{C}_{n+1}^{Q}(x;a)-(x-c_{n+1}^{qc})\mathcal{C}_{n}^{Q}(x;a)+\lambda_{n+1}^{qc}\mathcal{C}_{n-1}^{Q}(x;a)=0, ~ n\geq 1,
	\end{align}
	where the recurrence parameters are given by
	\begin{align}\label{quasi-Christoffel-recurrence-parameter}
		\lambda_{n+1}^{qc}=\frac{\gamma_{n}}{\gamma_{n-1}}\lambda_{n}^c, ~~~~~~~c_{n+1}^{qc}=c_{n+1}^c+\gamma_{n}-\gamma_{n+1}.
	\end{align}
	If $\lambda_{n+1}^{qc}\neq 0$, then  $\{\mathcal{C}^Q_n(x;a)\}$ forms a monic orthogonal polynomial sequence with respect to a certain measure, denoted by $\nu$.
\end{proposition}

Since $\mathcal{C}_{n}^Q(x;a)$ is a self-perturbation of the Christoffel polynomials, it is natural to establish a relation between the Jacobi matrix corresponding to the polynomials generated by $\mathcal{L}^C$
and the polynomial $\mathcal{C}_{n}^Q(x;a)$ satisfying the TTRR \eqref{TTRR_QCP}. Consequently, we observe that the matrix
$\mathcal{M}$ in \eqref{matrixform-quasiChristoffel} plays a crucial role in obtaining the relation between the Jacobi matrices.

\begin{theorem}
	Let $\mathcal{C}_n(x;a)$ be a monic Christoffel polynomial with recurrence parameters $\lambda_{n}^c$ and $c_n^c$ given in \eqref{Kernel recurrence parameters}. Let also that  $\mathcal{C}_n^Q(x;a)$ be a monic quasi-Christoffel polynomial of order one with recurrence parameters  $\lambda_{n}^{qc}$ and $c_n^{qc}$ given in \eqref{quasi-Christoffel-recurrence-parameter}. Then, the corresponding Jacobi matrices relate as follows:
	\begin{align}\label{}
		\left(\begin{array}{ccccc}
			c_1^{qc} & 1 & 0 & \cdots \\
			\lambda_2^{qc} & c_2^{qc} & 1 & \\
			0 & \lambda_3^{qc} & c_3^{qc} & 1 \\
			\vdots & \vdots & \ddots & \ddots
		\end{array}\right)\mathcal{M}	=\mathcal{M}\left(\begin{array}{ccccc}
			c_1^{c} & 1 & 0 & \cdots \\
			\lambda_2^{c} & c_2^{c} & 1 & \\
			0 & \lambda_3^{c} & c_3^{c} & 1 \\
			\vdots & \vdots & \ddots & \ddots
		\end{array}\right),
	\end{align}
	where matrix $\mathcal{M}$ is given in \eqref{matrixform-quasiChristoffel}.
\end{theorem}
\begin{proof}
	An eigenvalue equation can be obtained from \eqref{TTRR_QCP}, which can be written as:
	\begin{align}\label{EVequation_quasichristoffelpolynomials}
		\left(\begin{array}{ccccc}
			c_1^{qc} & 1 & 0 & \cdots \\
			\lambda_2^{qc} & c_2^{qc} & 1 & \\
			0 & \lambda_3^{qc} & c_3^{qc} & 1 \\
			\vdots & \vdots & \ddots & \ddots & \ddots
		\end{array}\right)\left(\begin{array}{c}
			\mathcal{C}^Q_{0}(x;a)\\
			\mathcal{C}^Q_{1}(x;a)\\
			\mathcal{C}^Q_{2}(x;a)\\
			\vdots
		\end{array}\right)=x	\left(\begin{array}{c}
			\mathcal{C}^Q_{0}(x;a)\\
			\mathcal{C}^Q_{1}(x;a)\\
			\mathcal{C}^Q_{2}(x;a)\\
			\vdots
		\end{array}\right).
	\end{align}
	We can write \eqref{Quasi-Christoffel order one} in the matrix form as:
	\begin{align}\label{matrixform-quasiChristoffel}
		\mathcal{M}\left(\begin{array}{c}
			\mathcal{C}_{0}(x;a)\\
			\mathcal{C}_{1}(x;a)\\
			\mathcal{C}_{2}(x;a)\\
			\vdots
		\end{array}\right)=	\left(\begin{array}{c}
			\mathcal{C}_{0}^Q(x;a)\\
			\mathcal{C}_{1}^Q(x;a)\\
			\mathcal{C}_{2}^Q(x;a)\\
			\vdots
		\end{array}\right),
	\end{align}
	where
	\begin{align}
		\mathcal{M}=\left(\begin{array}{cccccc}
			1 & 0 & 0&0&0& \cdots \\
			\gamma_1 & 1 & 0&0 &0&\cdot \\
			0 & \gamma_2 & 1& 0&\cdot \\
			\vdots & \vdots & \ddots & \ddots & \ddots&\ddots
		\end{array}\right).
	\end{align}
	By using \eqref{EVequation_Kernelpolynomials}, we obtain
	\begin{align}\label{matrixform_involve_Christoffel_quasiChristoffel}
		x\left(\begin{array}{c}
			\mathcal{C}_{0}^Q(x;a)\\
			\mathcal{C}_{1}^Q(x;a)\\
			\mathcal{C}_{2}^Q(x;a)\\
			\vdots
		\end{array}\right)=\mathcal{M}\left(\begin{array}{ccccc}
			c_1^{c} & 1 & 0 & \cdots \\
			\lambda_2^{c} & c_2^{c} & 1 & \\
			0 & \lambda_3^{c} & c_3^{c} & 1 \\
			\vdots & \vdots & \ddots & \ddots
		\end{array}\right)\left(\begin{array}{c}
			\mathcal{C}_{0}(x;a)\\
			\mathcal{C}_{1}(x;a)\\
			\mathcal{C}_{2}(x;a)\\
			\vdots
		\end{array}\right).
	\end{align}
	
	By substituting \eqref{matrixform-quasiChristoffel} and \eqref{matrixform_involve_Christoffel_quasiChristoffel} into \eqref{EVequation_quasichristoffelpolynomials}, we can obtain the relation between the Jacobi matrix corresponding to the Christoffel polynomials and the Jacobi matrix corresponding to the polynomials satisfying \eqref{TTRR_QCP}.
\end{proof}

In Subsection \ref{sub:QCJP}, we demonstrate that the quasi-Christoffel Jacobi polynomial of order one becomes orthogonal and derives its recurrence coefficients. This leads to the study of zeros of quasi-Christoffel Jacobi polynomial of order one before and after obtaining the orthogonality.   We also show that the measure associated with the quasi-Christoffel Jacobi polynomial of order one belongs to the Nevai class \cite{B.SimonSzegoDescendants}.
\subsection{ Quasi-Christoffel Jacobi Polynomial of order one }\label{sub:QCJP}Let $\mathcal{P}^{(\alpha,\beta)}_n(x)$ denote the monic Jacobi polynomials defined by the three-term recurrence relation:
\begin{align}
	\mathcal{P}^{(\alpha,\beta)}_{n+1}(x)=(x-c_{n+1})\mathcal{P}^{(\alpha,\beta)}_{n}(x)-\lambda_{n+1}\mathcal{P}^{(\alpha,\beta)}_{n-1}(x),
\end{align}
with recurrence coefficients given by
\begin{align*}
	\lambda_{n+1}&=\frac{4n(n+\alpha)(n+\beta)(n+\alpha+\beta)}{(2n+\alpha+\beta)^2(2n+\alpha+\beta+1)(2n+\alpha+\beta-1)},\\
	c_{n+1}&=\frac{\beta^2-\alpha^2}{(2n+\alpha+\beta)(2n+\alpha+\beta+2)}.
\end{align*}
The Jacobi polynomials form an orthogonal sequence on the interval $(-1,1)$ with respect to the weight function $w(x)=(1-x)^\alpha(1+x)^\beta$, where $\alpha>-1$ and $\beta>-1$ (see \cite{Chihara book}). The significance of Jacobi polynomials, including their special instances like ultraspherical polynomials, Legendre polynomials and  Chebyshev polynomials extends across various mathematical domains. One notable application lies in their connection to the spectral analysis of Laplacian and sub-Laplacian operators. Pertaining to this, \cite{Casarino_Martini_Pointwise estimates_Jacobi_CR Acad_2021} delves into the discussion on pointwise estimations for ultraspherical polynomials. The exploration of Pell's equation as it relates to Chebyshev polynomials is detailed in \cite{Lasserre_Pell's eqn_CR Acad_2023}. Moreover, these orthogonal polynomials are intricately connected to convex optimization and real algebraic geometry, as elucidated in the same source \cite{Lasserre_Pell's eqn_CR Acad_2023}.

Upon applying the Christoffel transformation at $a=-1$ to this weight function, we obtain $\tilde{w}(x)=(1-x)^\alpha(1+x)^{\beta+1}$, with $\alpha>-1$ and $\beta>-1$. The Christoffel transformed polynomial of the Jacobi polynomial is another Jacobi polynomial with parameter $(\alpha,\beta+1)$. This  Christoffel Jacobi polynomial with parameter $(\alpha,\beta+1)$ is denoted by $\mathcal{C}_n(x;-1):=\mathcal{P}^{(\alpha,\beta+1)}_n(x)$, and is expressed as:
\begin{align}\label{Christoffel-Jacobi poly}
\mathcal{C}_{n}(x;-1):=\mathcal{P}^{(\alpha,\beta+1)}_{n}(x)=\frac{1}{x+1}\left(\mathcal{P}^{(\alpha,\beta)}_{n+1}(x)+\frac{2(\beta+n)(n+\alpha+\beta)}{(2n+\alpha+\beta)(2n+\alpha+\beta-1)}\mathcal{P}^{(\alpha,\beta)}_{n}(x)\right).
\end{align}
This satisfies the three-term recurrence relation:
\begin{align*}
	\mathcal{P}^{(\alpha,\beta+1)}_{n+1}(x)=(x-c^c_{n+1})\mathcal{P}^{(\alpha,\beta+1)}_{n}(x)-\lambda^c_{n+1}\mathcal{P}^{(\alpha,\beta+1)}_{n-1}(x),
\end{align*}
where the transformed recurrence parameters are given by:
\begin{align*}
	\lambda^c_{n+1}&=\frac{4n(n+\alpha)(n+\beta+1)(n+\alpha+\beta+1)}{(2n+\alpha+\beta+1)^2(2n+\alpha+\beta+2)(2n+\alpha+\beta)},\\
	c^c_{n+1}&=\frac{(\beta+1)^2-\alpha^2}{(2n+\alpha+\beta+1)(2n+\alpha+\beta+3)}.
\end{align*}
By forming a linear combination of two consecutive degrees of Christoffel Jacobi polynomials, we define the quasi-Christoffel Jacobi polynomial of order one as:
\begin{align}\label{quasi-Christoffel Jacobi poly.}
	J^{QC}_{n}(x;-1)=\mathcal{P}^{(\alpha,\beta+1)}_{n}(x)+\gamma_{n}\mathcal{P}^{(\alpha,\beta+1)}_{n-1}(x).
\end{align}
Subsequently, our aim is to obtain the orthogonality of $J^{QC}_{n}(x;-1)$ defined in \eqref{quasi-Christoffel Jacobi poly.} and to write the compact structure of the polynomial $J^{QC}_{n}(x;-1)$.
\subsubsection{\underline{Orthogonality of Quasi-Christoffel Jacobi polynomials, $J^{QC}_{n}(x;-1)$}}
To achieve the orthogonality of $J^{QC}_{n}(x;-1)$, we must determine the value of $\gamma_n$ for $n=2,3,...$ that satisfies \eqref{gamma n restriction cond}. We have
\begin{align}\label{gamma condition for Jacobi}
	\nonumber	&-\frac{(\beta+1)^2-\alpha^2}{(\alpha+\beta+2 n-1) (\alpha+\beta+2 n+1)}+\frac{(\beta+1)^2-\alpha^2}{(\alpha+\beta+2 n+1) (\alpha+\beta+2 n+3)}+\gamma_n-\gamma_{n+1}\\
	\nonumber&\hspace{1cm}	-\frac{1}{\gamma_{n}}\frac{4 n (\alpha+n) (\beta+n+1) (\alpha+\beta+n+1)}{ (\alpha+\beta+2 n) (\alpha+\beta+2
		n+1)^2 (\alpha+\beta+2 n+2)}\\
	&\hspace{3cm}+\frac{1}{\gamma_{n-1}}\frac{4 (n-1) (\alpha+n-1) (\beta+n) (\alpha+\beta+n)}{(\alpha+\beta+2 n-2) (\alpha+\beta+2 n-1)^2 (\alpha+\beta+2 n)}=0.
\end{align}
We provide four possible solutions to \eqref{gamma condition for Jacobi}, which are as follows:

\label{Sol1-QCJP}\textbf{Solution 1.} Recursively, we obtain the explicit expression for $\gamma_{n}$ for $n=1,2,3,...$ as follows:
\begin{align}\label{Explicit gamma value for QCJP_Sol1}
	\gamma_{n}=-\frac{2(\alpha+n)(n+\alpha+\beta+1)}{(2n+\alpha+\beta+1)(2n+\alpha+\beta)},
\end{align}
which satisfies the nonlinear difference equation \eqref{gamma condition for Jacobi}. 
Thus, the polynomial $J^{QC}_{n}(x;-1)$ for $n=1,2,...$ is defined as:
\begin{align}\label{quasi-Christoffel Jacobi orthogonal}
	J^{QC}_{n}(x;-1)=\mathcal{P}^{(\alpha,\beta+1)}_{n}(x)-\frac{2(\alpha+n)(n+\alpha+\beta+1)}{(2n+\alpha+\beta+1)(2n+\alpha+\beta)}\mathcal{P}^{(\alpha,\beta+1)}_{n-1}(x),
\end{align}
which becomes orthogonal and satisfies the three-term recurrence relation
\begin{align}\label{TTRR_Quasi-Chris_Jacobi_Sol1}
	J^{QC}_{n+1}(x;-1)=(x-c^{qc}_{n+1})J^{QC}_{n}(x;-1)-\lambda^{qc}_{n+1}J^{QC}_{n-1}(x;-1),
\end{align} with recurrence coefficients given by:
\begin{align}\label{recurrence coef. quasi Chris Jacobi}
	\nonumber	\lambda^{qc}_{n+1}&=\frac{4(n-1)(n+\alpha)(n+\beta)(n+\alpha+\beta+1)}{(2n+\alpha+\beta)^2(2n+\alpha+\beta+1)(2n+\alpha+\beta-1)},\\
	c^{qc}_{n+1}&=\frac{(\beta-\alpha)(2+\beta+\alpha)}{(2n+\alpha+\beta)(2n+\alpha+\beta+2)}.
\end{align}
Hence, the sequence of polynomials $\{J^{QC}_{n}(x;-1)\}_{n=2}^{\infty}$ becomes orthogonal with respect to a certain measure.  We can obtain the compact form of \eqref{quasi-Christoffel Jacobi poly.} in terms of the Jacobi family by using the common factor $x-1$ of the polynomial $J^{QC}_{n}(x;-1)$. Consequently, one may write \eqref{quasi-Christoffel Jacobi poly.} as follows:
{\small\begin{align}\label{compactform_QCJOP_Sol1}
		J^{QC}_{n}(x;-1)=(x-1)\mathcal{P}^{(\alpha+1,\beta+1)}_{n-1}(x)=\mathcal{P}^{(\alpha,\beta+1)}_{n}(x)-\frac{2(\alpha+n)(n+\alpha+\beta+1)}{(2n+\alpha+\beta+1)(2n+\alpha+\beta)}\mathcal{P}^{(\alpha,\beta+1)}_{n-1}(x).
\end{align}}
The presence of a zero on the boundary of the interval of orthogonality ensures that the corresponding linear functional is not positive definite. Thus, the polynomial $J^{QC}_{n}(x;-1)$ for $n\geq1$ defined in \eqref{quasi-Christoffel Jacobi poly.} with $\gamma_{n}$ from \eqref{Explicit gamma value for QCJP_Sol1} becomes orthogonal with respect to the Geronimus measure $d\mu_1(x) = (1-x)^{\alpha-1}(1+x)^{\beta+1}dx, \alpha>-1, \beta>-1$.

\begin{remark}[Nevai class]
	The  monic Jacobi matrix corresponding to \eqref{TTRR_Quasi-Chris_Jacobi_Sol1}  can be expressed as:
	\begin{align}\label{Jacobi_matrix_quasi_Chris_Jacobi}
		\mathcal{J}^J_{qc}=\left(\begin{array}{ccccc}
			c_1^{qc} & 1 & 0 & \cdots \\
			\lambda_2^{qc} & c_2^{qc} & 1 & \\
			0 & \lambda_3^{qc} & c_3^{qc} & 1 \\
			\vdots & \vdots & \ddots & \ddots & \ddots
		\end{array}\right),
	\end{align}
	where the entries of the matrix are defined in \eqref{recurrence coef. quasi Chris Jacobi}.
	From \eqref{recurrence coef. quasi Chris Jacobi}, it is clear that the recurrence coefficients $\lambda^{qc}_j$ and $c^{qc}_j$ are bounded. This implies that the monic Jacobi matrix $\mathcal{J}^J_{qc}$ is bounded and all the zeros of the orthogonal polynomials $J^{QC}_{n}(x;-1)$ defined as \eqref{quasi-Christoffel Jacobi orthogonal} are bounded and lie in $(-1,1]$. For any $\alpha, \beta \in \mathbb{R}$ with $\alpha>0$, a measure is said to belong to the Nevai class, $\mathbb{N}(\alpha,\beta)$, if the recurrence parameters $\lambda_n \rightarrow \alpha$ and $c_n\rightarrow\beta$ as $n\rightarrow\infty$ (see \cite{Lubinsky_SCM_NC_ProcAMS_1991}).  Moreover, the recurrence parameters $\lambda^{qc}_{n+1}$ and $c^{qc}_{n+1}$ defined in \eqref{recurrence coef. quasi Chris Jacobi} converge to $1/4$ and $0$ respectively as $n\rightarrow\infty$. This implies that the measure associated with $\mathcal{J}^J_{qc}$ is in the Nevai class $\mathbb{N}(1/4,0)$. For more information about the Nevai class, we refer to \cite{B.SimonSzegoDescendants,Lubinsky_SCM_NC_ProcAMS_1991}.
	\end{remark}

\textbf{Solution 2.} Another solution of the nonlinear difference equation \eqref{gamma condition for Jacobi} is given by:
\begin{align}\label{Explicit gamma value for QCJP_Sol2}
	\gamma_{n}=\frac{2(\beta+n+1)(n+\alpha+\beta+1)}{(2n+\alpha+\beta+1)(2n+\alpha+\beta)}.
\end{align}
Using this value of $\gamma_n$, the polynomial $J^{QC}_{n}(x;-1)$ for $n=1,2,...$ is defined as:
{\small\begin{align}\label{quasi-Christoffel Jacobi orthogonal_Sol2}
		J^{QC}_{n}(x;-1):=(x+1)\mathcal{P}^{(\alpha,\beta+2)}_{n-1}(x)=\mathcal{P}^{(\alpha,\beta+1)}_{n}(x)+\frac{2(\beta+n+1)(n+\alpha+\beta+1)}{(2n+\alpha+\beta+1)(2n+\alpha+\beta)}\mathcal{P}^{(\alpha,\beta+1)}_{n-1}(x),
\end{align}}
becomes orthogonal and satisfies the three-term recurrence relation
\begin{align}\label{TTRR_Quasi-Chris_Jacobi_Sol2}
	J^{QC}_{n+1}(x;-1)=(x-c^{qc}_{n+1})J^{QC}_n(x;-1)-\lambda^{qc}_{n+1}J^{QC}_{n-1}(x;-1),
\end{align} with recurrence coefficients given by:
\begin{align}\label{recurrence coef. quasi Chris Jacobi_Sol2}
	\nonumber	\lambda^{qc}_{n+1}&=\frac{4(n-1)(n+\alpha-1)(n+\beta+1)(n+\alpha+\beta+1)}{(2n+\alpha+\beta)^2(2n+\alpha+\beta+1)(2n+\alpha+\beta-1)},\\
	c^{qc}_{n+1}&=\frac{(\beta+2)^2-\alpha^2}{(2n+\alpha+\beta)(2n+\alpha+\beta+2)}.
\end{align}
Thus, the orthogonality measure for the polynomial $J^{QC}_{n}(x;-1)$ for $n\geq1$ as defined in \eqref{quasi-Christoffel Jacobi orthogonal_Sol2} is $d\mu_3(x)=(1-x)^{\alpha}(1+x)^{\beta}dx$, where $\alpha>-1, \beta>-1$.

\textbf{Solution 3.} Considering \begin{align}\label{nonlinear_difference_Sol3}
	\gamma_{n} = \frac{2n(\alpha+n)}{(2n+\alpha+\beta+1)(2n+\alpha+\beta)},
\end{align}
solves equation \eqref{gamma condition for Jacobi} for $\alpha > -1$ and $\beta > -1$. With this $\gamma_n$, the polynomial $J^{QC}_{n}(x;-1)$ becomes orthogonal and reduces to the original orthogonal polynomial:
\begin{align}\label{Jacobi interms extend parameter Jacobi_Sol3}
	J^{QC}_{n}(x;-1) := \mathcal{P}^{(\alpha,\beta)}_{n}(x) = \mathcal{P}^{(\alpha,\beta+1)}_{n}(x) + \frac{2n(\alpha+n)}{(2n+\alpha+\beta+1)(2n+\alpha+\beta)} \mathcal{P}^{(\alpha,\beta+1)}_{n-1}(x).
\end{align}
The corresponding recurrence parameters are:
\begin{align*}
	\lambda^{qc}_{n+1} = \lambda_{n+1}, \quad c^{qc}_{n+1} = c_{n+1}.
\end{align*}
The polynomial $J^{QC}_{n}(x;-1)$ for $n\geq0$ defined in \eqref{quasi-Christoffel Jacobi poly.} with $\gamma_{n}$ from \eqref{nonlinear_difference_Sol3} is orthogonal with respect to the measure, $d\mu_3(x)=(1-x)^{\alpha}(1+x)^{\beta}dx, \alpha>-1, \beta>-1$.

\textbf{Solution 4.} The parameter
\begin{align}
	\gamma_{n} = -\frac{2n(\beta+n+1)}{(2n+\alpha+\beta+1)(2n+\alpha+\beta)},
\end{align}
satisfies \eqref{gamma condition for Jacobi}. Thus the polynomial $J^{QC}_{n}(x;-1)$ defined in \eqref{quasi-Christoffel Jacobi poly.} can be written as:
\begin{align}\label{QCJOP_Sol4}
	J^{QC}_{n}(x;-1) := \mathcal{P}^{(\alpha-1,\beta+1)}_{n}(x) = \mathcal{P}^{(\alpha,\beta+1)}_{n}(x) - \frac{2n(\beta+n+1)}{(2n+\alpha+\beta+1)(2n+\alpha+\beta)} \mathcal{P}^{(\alpha,\beta+1)}_{n-1}(x),
\end{align}
and the recurrence parameters are given by
\begin{align}\label{recurrence coef. quasi Chris Jacobi_Sol4}
	\nonumber	\lambda^{qc}_{n+1}&=\frac{4n(n+\alpha-1)(n+\beta+1)(n+\alpha+\beta)}{(2n+\alpha+\beta)^2(2n+\alpha+\beta+1)(2n+\alpha+\beta-1)},\\
	c^{qc}_{n+1}&=\frac{(\beta+1)^2-(\alpha-1)^2}{(2n+\alpha+\beta)(2n+\alpha+\beta+2)}.
\end{align}
Hence, the orthogonality measure for the polynomial $J^{QC}_{n}(x;-1)$ for $n\geq1$ as defined in \eqref{QCJOP_Sol4} is $d\mu(x)=(1-x)^{\alpha-1}(1+x)^{\beta+1}dx$, where $\alpha>-1, \beta>-1$.

\begin{remark}
	The representation \eqref{Jacobi interms extend parameter Jacobi_Sol3} can also be seen as a decomposition of the Jacobi polynomial with parameter $(\alpha,\beta)$ in terms of the Jacobi polynomial with extended parameters $(\alpha,\beta+1)$. This representation is known as the connection formula in the literature, see \rm\cite[page 53]{Paco Book_2021}.
\end{remark}
\begin{remark}
	Similarly, we can define the quasi-Christoffel Jacobi polynomial of order one at $a = 1$, denoted by $J^{QC}_n(x;1)$, and derive four distinct solutions to the nonlinear difference equation \eqref{gamma n restriction cond}. These solutions correspond to the four different methods of establishing the orthogonality of $J^{QC}_n(x;1)$.
\end{remark}
\begin{remark}\label{explanation-four-solution-QCJOP}
 The non-linear difference equation \eqref{gamma condition for Jacobi} can have four possible solutions that lead to the compact form of $J^{QC}_{n}(x;-1)$ defined in \eqref{quasi-Christoffel Jacobi poly.}. These polynomials are listed as solutions 1, 2, 3, and 4 in this subsection. The compact form of these polynomials also provides an explicit expression for the measure. It is expected that any other solution of \eqref{gamma condition for Jacobi} can be represented as a superposition of Christoffel and Geronimus transformations of these four solutions. For more details in this direction, see \cite{Zhedanov_RST_OP_JCAM_1997}.
\end{remark}


\subsubsection{\underline{Zeros of quasi-Christoffel Jacobi polynomials}}

Due to quasi-orthogonality of the polynomial $J^{QC}_{n}(x;-1)$ defined in \eqref{quasi-Christoffel Jacobi poly.}, it is possible for some zeros to lie outside the support of a Jacobi measure. Subsequently, we observe numerically that at most one zero lies outside the support of the measure for the Jacobi polynomials. Furthermore, we note that a zero can be situated on either side of the support of a measure, either to the left or right.
\begin{table}[ht]
	\begin{center}
		\resizebox{!}{1.6cm}{\begin{tabular}{|c|c|c|c|}
				\hline
				\multicolumn{2}{|c|}{Zeros of $J^{QC}_5(x;-1)$}&\multicolumn{2}{|c|}{Zeros of $J^{QC}_6(x;-1)$}\\
				\hline
				$\alpha=-0.5$, $\beta=0$, $\gamma_n=3$ & $\alpha=0$, $\beta=0.5$, $\gamma_n=2$& $\alpha=1$, $\beta=-0.5$, $\gamma_n=-1$ &$\alpha=0.5$, $\beta=1$, $\gamma_n=-2$\\
				\hline
				-1.23179&-2.09864&-0.88766&-0.73675\\
				\hline
				-0.60752&-0.62066&-0.57465&-0.34365\\
				\hline
				-0.00608&-0.04931&-0.12792&0.11967\\
				\hline
				0.59528&0.51835&0.35637&0.56019\\
				\hline
				0.95223&0.90244&0.77089&0.88114\\
				\hline
				-&-&1.24075&2.14008\\
				\hline
		\end{tabular}}
		\captionof{table}{Zeros of  $J^{QC}_n(x;-1)$ in \eqref{quasi-Christoffel Jacobi poly.}}
		\label{Zeros_Quasi Christoffel Jacobi_outside}
	\end{center}
\end{table}

Table \ref{Zeros_Quasi Christoffel Jacobi_outside} demonstrates that for $\alpha=-0.5$ and $\beta=0$ with $\gamma_{n}=3$, one zero $(x_0=-1.23179)$ of $J^{QC}_5(x;-1)$ lies outside the left side of the interval $(-1,1)$. Similarly, for $\alpha=0.5$ and $\beta=1$ with $\gamma_{n}=-2$, one zero $(x_0=2.14008)$ of $J^{QC}_5(x;-1)$ lies outside the right side of the interval $(-1,1)$. These instances are illustrated in Table \ref{Zeros_Quasi Christoffel Jacobi_outside} for various values of $\alpha$, $\beta$, and $\gamma_n$. In fact, this behavior of zeros also holds for the general quasi-Christoffel polynomial of order one when the point $a$ lies strictly left to  the interval of orthogonality. Subsequently, using the technique discussed in \cite{Joulak_Quasi_orthogonal_2005_ANM}, we prove the following:
\begin{proposition}\label{quasi-Chris_zerooutside_right_IOO}
	Let $\{\mathcal{C}_n(x;a)\}$ be a sequence of orthogonal polynomials with respect to the positive definite linear functional $\mathcal{L}^C$. The necessary and sufficient condition for exactly one zero of $\mathcal{C}^Q_n(x;a)$, defined in \eqref{Quasi-Christoffel order one}, to lie to the right of the interval of orthogonality $(c,d)$ is that:
	\begin{align}\label{cond-quasi-Chris_zerooutside_right_IOO}
		\gamma_n<-\frac{\mathcal{C}_n(d;a)}{\mathcal{C}_{n-1}(d;a)}<0.
	\end{align}
\end{proposition}
\begin{proof}
	Suppose $z_1,z_2,...,z_n$, with $z_1<z_2<...<z_{n-1}<z_n$,  vanish the polynomial $\mathcal{C}^Q_n(x;a)$ of degree $n$. If $d<z_n$, then, using the fact that $z_n$ is the largest zero of $\mathcal{C}^Q_n(x;a)$, we obtain $\mathcal{C}^Q_n(d;a)<0\implies \gamma_{n}< -\frac{\mathcal{C}_n(d;a)}{\mathcal{C}_{n-1}(d;a)}$. Since $\mathcal{C}_n(d;a)$ and $\mathcal{C}_{n-1}(d;a)$ share the same positive sign, we get \eqref{cond-quasi-Chris_zerooutside_right_IOO}. Conversely, if \eqref{cond-quasi-Chris_zerooutside_right_IOO} holds, then the positivity of $\mathcal{C}^Q_n(x;a)$ for the large values of $x>d$ ensures that there exists $z_n>d$ at which $\mathcal{C}^Q_n(x;a)$ vanishes.
\end{proof}
\begin{remark}
	In a similar manner to Proposition \ref{quasi-Chris_zerooutside_right_IOO}, we see that the condition on $\gamma_n$, given by $	\gamma_n>-\frac{\mathcal{C}_n(c;a)}{\mathcal{C}_{n-1}(c;a)}>0$ is necessary and sufficient for obtaining one zero of $\mathcal{C}^Q_n(x;a)$  that lies to the left of the interval of orthogonality.
\end{remark}


  \begin{table}[ht]
	\begin{center}
		\resizebox{!}{2cm}{\begin{tabular}{|c|c|c|c|}
				\hline
				\multicolumn{2}{|c|}{Zeros of $J^{QC}_n(x;-1)$}&\multicolumn{2}{|c|}{Zeros of $J^{QC}_n(x;-1)$}\\
				\hline
				$n=7$,	$\alpha=0.1$, $\beta=-0.4$ &$n=8$, $\alpha=0.1$, $\beta=-0.4$&$n=9$, $\alpha=1.3$, $\beta=0.4$ &$n=10$, $\alpha=1.3$, $\beta=0.4$\\
				\hline
				-0.901465&-0.923446&-0.911302&-0.926224\\
				\hline
			-0.639281&-0.716709&-0.73988&-0.782531\\
				\hline
				-0.261342&-0.409266&-0.500104&-0.578782\\
				\hline
			0.164331&-0.0441403&-0.213879&-0.330473\\
				\hline
				0.561137&0.327586&0.0926373&-0.0565031\\
				\hline
				0.857643&0.653905&0.391443&0.22228\\
				\hline
				1&0.88914&0.655282&0.484672\\
				\hline
				-&1&0.860322&0.710749\\
				\hline
				-&-&1&0.88354\\
				\hline
				-&-&-&1\\
				\hline
				\end{tabular}}
		\captionof{table}{Zeros of  $J^{QC}_n(x;-1)$ with $\gamma_{n}$ given in \eqref{Explicit gamma value for QCJP_Sol1} }
		\label{Zeros_Quasi Christoffel Orthogonal Jacobi}
	\end{center}
\end{table}

Next, we observe that using the value of $\gamma_{n}$ mentioned in \eqref{Explicit gamma value for QCJP_Sol1} yields the orthogonality of the polynomial $J^{QC}_n(x;-1)$ defined in \eqref{quasi-Christoffel Jacobi poly.}. The behavior of the zeros of $J^{QC}_n(x;-1)$ is illustrated in Table \ref{Zeros_Quasi Christoffel Orthogonal Jacobi}. For each $n\geq1$, it is also noted that $p(x)=x-1$ is a factor of the polynomial $J^{QC}_n(x;-1)$. This implies that exactly one zero, $x=1$, lies on the boundary of the true interval of orthogonality for Jacobi polynomials.

\begin{figure}[!ht]
	\includegraphics[scale=0.7]{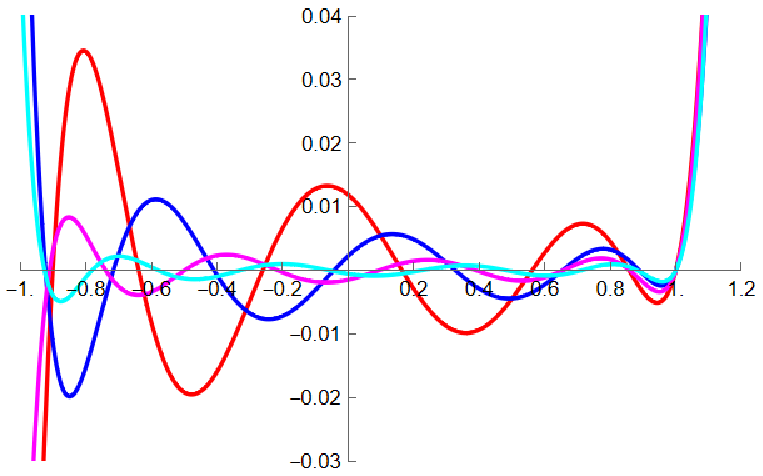}
	\caption{Graph of $J^{QC}_{n}(x;-1)$  with $\gamma_{n}$ as \eqref{Explicit gamma value for QCJP_Sol1} and $(n,\alpha,\beta)$: (7,0.1,-0.4)(Red), (8,0.1,-0.4) (blue), (9,1.3,0.4) (Magenta) and (10,1.3,0.4) (Cyan).}
	\label{Interlace_Zeros_quasi-Christoffel orthogonal Jacobi}
\end{figure}

In Table \ref{Zeros_Quasi Christoffel Orthogonal Jacobi} and Figure \ref{Interlace_Zeros_quasi-Christoffel orthogonal Jacobi}, it is observed that for a particular set of $\alpha$ and $\beta$ values, such as $\alpha=1.3$ and $\beta=0.4$, the zeros of $J^{QC}_{9}(x;-1)$ and $J^{QC}_{10}(x;-1)$ exhibit interlacing behaviour. Additionally, Table \ref{Zeros_Quasi Christoffel Orthogonal Jacobi} illustrates that exactly one zero lies on the boundary of the interval $(-1,1)$ while all others lie within the interval.
The zeros of $\mathcal{P}_n^{(\alpha,\beta)}(x)$, $\mathcal{C}_{n}(x;-1)$, and $J^{QC}_{n}(x;-1)$, with $\gamma_{n}$ given in \eqref{Explicit gamma value for QCJP_Sol1}, exhibit an interlacing pattern. We verified this triple interlacing for various values of $n, \alpha$, and $\beta$. Specifically, we demonstrated the triple interlacing for $n=7, \alpha=1.3$, and $\beta=-0.6$ in Figure \ref{n7 Interlace_Zeros_Jacobi_CJ_QCJOP} and Table \ref{n7 Zeros_Jacobi_Christoffel Jacobi_Quasi Christoffel Jacobi}. For $n=8, \alpha=1.3$, and $\beta=-0.6$, this pattern is shown in Figure \ref{n8 Interlace_Zeros_Jacobi_CJ_QCJOP} and Table \ref{n8 Zeros_Jacobi_Christoffel Jacobi_Quasi Christoffel Jacobi}.

\noindent\begin{minipage}{1.0\linewidth}
	\hspace{0cm}	
	\noindent	\begin{minipage}{0.5\linewidth}
		\begin{figure}[H]
			\includegraphics[width=\linewidth]{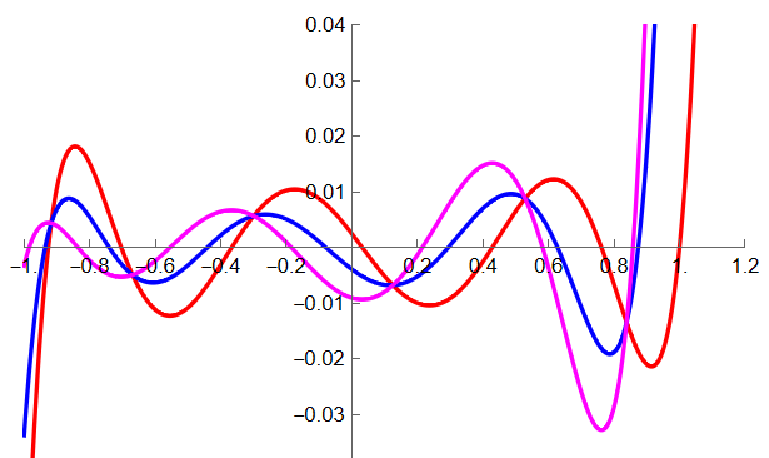}
			\captionof{figure}{Graph of $\mathcal{P}_7^{(1.3,-0.6)}(x)$ (Magenta), $\mathcal{C}_7(x;-1)$ (Blue) and  $J^{QC}_7(x;-1)$(red) with $\gamma_{n}$ given in \eqref{Explicit gamma value for QCJP_Sol1} .}
			\label{n7 Interlace_Zeros_Jacobi_CJ_QCJOP}
		\end{figure}
	\end{minipage}
	\noindent	\begin{minipage}{0.40\linewidth}
		\begin{table}[H]
			\resizebox{!}{1.6cm}{\begin{tabular}{|c|c|c|}
					\hline
					\multicolumn{1}{|c|}{Zeros of $\mathcal{P}^{(\alpha,\beta)}_n(x)$}&	\multicolumn{1}{|c|}{Zeros of $\mathcal{C}_n(x;-1)$}&\multicolumn{1}{|c|}{Zeros of $J^{QC}_n(x;-1)$}\\
					\hline
					\multicolumn{3}{|c|}{
						$n=7$,	$\alpha=1.3$, $\beta=-0.6$ }\\
					\hline
					-0.98451&-0.935875&-0.926421\\
					\hline
					-0.836149&-0.740835&-0.703976\\
					\hline
					-0.554702&-0.441406&-0.36693\\
					\hline
					-0.184727&-0.0794614&0.0315102\\
					\hline
					0.2153&0.294368&0.428496\\
					\hline
					0.582157&0.627786&0.761811\\
					\hline
					0.857868&0.874149&1\\
					\hline
			\end{tabular}}
			\captionof{table}{Zeros of\\  $J^{QC}_7(x;-1)$ with $\gamma_{n}$ given in \eqref{Explicit gamma value for QCJP_Sol1} }
			\label{n7 Zeros_Jacobi_Christoffel Jacobi_Quasi Christoffel Jacobi}
		\end{table}
	\end{minipage}
\end{minipage}


For $\alpha>-1$ and $\beta>-1$, the zeros of the Jacobi polynomials $\mathcal{P}_n^{(\alpha,\beta)}(x)$ are located inside the interval $(-1,1)$ and are simple. However, when we extend the values of the parameters $(\alpha,\beta)$, this result no longer holds and also the Jacobi polynomials no longer maintain orthogonality.  Nonetheless, by discarding some initial terms of the Jacobi polynomial sequence, orthogonality can still be achieved. Specifically, for $\alpha=-1$ and $\beta>-1$, $x=1$ becomes the common zero of all the polynomials $\mathcal{P}_n^{(-1,\beta)}(x)$ and to achieve the standard orthogonality of Jacobi polynomials, it is necessary to eliminate the first term of the Jacobi sequence. Thus, the sequence $\{\mathcal{P}_n^{(-1,\beta)}(x)\}_{n=1}^{\infty}$ becomes orthogonal under the Jacobi measure \cite{Littlejohn_2012_Jacobi_Sobolev}. More generally, for the parameters $\alpha=-k$ where $k\in \mathbb{N}$ and $\beta>-1$, orthogonality of the Jacobi polynomials $\{\mathcal{P}_n^{(-k,\beta)}(x)\}_{n=k}^{\infty}$ can be achieved by eliminating the first $k$ terms of the Jacobi sequence. This can be seen from the following formula (see \cite[equation 4.22.2]{Szego}):
\begin{align}\label{Jacobi relation negative and positive parameter}
	\mathcal{P}_n^{(-k,\beta)}(x)=\frac{1}{2^k}\frac{\Gamma(n+\beta+1)\Gamma(n-k+1)}{\Gamma(n+\beta+1-k)\Gamma(n+1)}(x-1)^k\mathcal{P}_{n-k}^{(k,\beta)}(x).
\end{align}
When seeking orthogonality of Jacobi polynomials for negative integral values of $\alpha$, utilizing \eqref{Jacobi relation negative and positive parameter}, we encounter a multiplicity $m$ for the zero $x=1$. However, the polynomial obtained by restoring orthogonality as described in equation \eqref{quasi-Christoffel Jacobi orthogonal} has the common zero $x=1$ with multiplicity one, which remains independent of the parameters $\alpha$ and $\beta$. The independence from the parameters $\alpha$ and $\beta$ of the factor $p(x)=x-1$ is also evident in table \ref{Zeros_Quasi Christoffel Orthogonal Jacobi} and \ref{n8 Zeros_Jacobi_Christoffel Jacobi_Quasi Christoffel Jacobi}.

\begin{minipage}{1.0\linewidth}
	\noindent	\begin{minipage}{0.5\linewidth}
		\begin{figure}[H]
			\includegraphics[width=\linewidth]{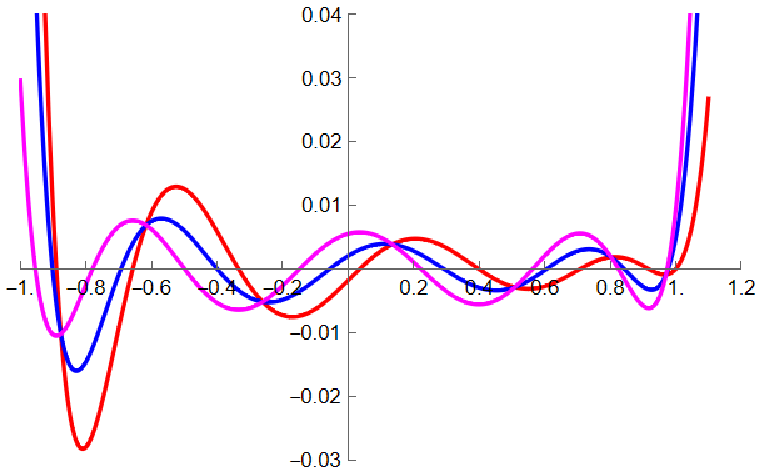}
			\captionof{figure}{Graph of $\mathcal{P}_8^{(-0.3,0.1)}(x)$ (Magenta), $\mathcal{C}_8(x;-1)$ (Blue) and  $J^{QC}_8(x;-1)$(red) with $\gamma_{n}$ given in \eqref{Explicit gamma value for QCJP_Sol1}}.
			\label{n8 Interlace_Zeros_Jacobi_CJ_QCJOP}
		\end{figure}
	\end{minipage}
	\noindent	\begin{minipage}{0.40\linewidth}
		\begin{table}[H]
			\resizebox{!}{1.7cm}{\begin{tabular}{|c|c|c|}
					\hline
					\multicolumn{1}{|c|}{Zeros of $\mathcal{P}^{(\alpha,\beta)}_n(x)$}&	\multicolumn{1}{|c|}{Zeros of $\mathcal{C}_n(x;-1)$}&\multicolumn{1}{|c|}{Zeros of $J^{QC}_n(x;-1)$}\\
					\hline
					\multicolumn{3}{|c|}{
						$n=8$,	$\alpha=-0.3$, $\beta=0.1$ }\\
					\hline
					-0.954048&-0.902299&-0.890383\\
					\hline
					-0.780467&-0.693404&-0.657609\\
					\hline
					-0.499018&-0.39941&-0.334239\\
					\hline
					-0.148551&-0.0565928&0.0349937\\
					\hline
					0.22249&0.29277&0.399036\\
					\hline
					0.562812&0.605597&0.707552\\
					\hline
					0.825373&0.84332&0.917865\\
					\hline
					0.973941&0.976685&1\\
					\hline
			\end{tabular}}
			\captionof{table}{Zeros of\\  $J^{QC}_8(x;-1)$ with $\gamma_{n}$ given in \eqref{Explicit gamma value for QCJP_Sol1} }
			\label{n8 Zeros_Jacobi_Christoffel Jacobi_Quasi Christoffel Jacobi}
		\end{table}
	\end{minipage}
\end{minipage}
\begin{remark}
It may be noted that the orthogonality of $J^{QC}_{n}(x;-1)$ achieved in solutions 2, 3, and 4 is again similar to the family of Jacobi orthogonal polynomials with varying parameters, whose numerical illustrations are abundant in the literature (see \cite{Driver_Jordaan_Mbuyi_Zeros-Jacobi-differentparam-numalg_2008,Driver_Littlejohn_zeros_Jacobi-2023}). Hence, we have done the analysis of zeros for the polynomials \eqref{compactform_QCJOP_Sol1} obtained in solution 1, but not for solutions 2, 3, and 4 in Subsection \ref{Sol1-QCJP}.
\end{remark}

\subsection{Quasi-Christoffel Laguerre polynomial} The monic Laguerre polynomials are characterized by the following three-term recurrence relation \cite[page 154]{Chihara book}:
\begin{align}\label{Laguerre_TTRR}
	\mathcal{L}^{(\alpha)}_{n+1}(x)=(x-c_{n+1})\mathcal{L}^{(\alpha)}_{n}(x)-\lambda_{n+1}\mathcal{L}^{(\alpha)}_{n-1}(x),
\end{align}
with initial data $\mathcal{L}^{(\alpha)}_{-1}(x)=0,     \mathcal{L}^{(\alpha)}_{0}(x)=1$. The recurrence parameters, denoted by $c_{n+1}$ and $\lambda_{n+1}$, are expressed as $c_{n+1}=2n+\alpha+1$ and $\lambda_{n+1}=n(n+\alpha)$. These Laguerre polynomials exhibit orthogonality within the interval $(0,\infty)$ concerning the weight function $w(x;\alpha)=x^{\alpha}e^{-x}$, where $\alpha>-1$. Upon applying the Christoffel transformation to the Laguerre weight with $a=0$, the resulting transformed weight is $\tilde{w}(x;\alpha)=x^{\alpha+1}e^{-x}$, $\alpha>-2$. Consequently, the Christoffel Laguerre polynomials at $a=0$ assume the form of the Laguerre polynomial with parameter $\alpha+1$.  The monic Christoffel Laguerre polynomials at $a=0$, denoted by $\mathcal{C}_n(x;0):=\mathcal{L}^{(\alpha+1)}_{n}(x)$, are generated by the three-term recurrence relation
\begin{align}
	\mathcal{L}^{(\alpha+1)}_{n+1}(x)=(x-c^c_{n+1})\mathcal{L}^{(\alpha+1)}_{n}(x)-\lambda^c_{n+1}\mathcal{L}^{(\alpha+1)}_{n-1}(x),
\end{align}
with initial conditions  $\mathcal{L}^{(\alpha+1)}_{-1}(x)=0, \mathcal{L}^{(\alpha+1)}_{0}(x)=1$. The recurrence coefficients are  $c^c_{n+1}=2n+\alpha+2$ and $\lambda^c_{n+1}=n(n+\alpha+1)$.
 \newline
The monic quasi-Christoffel Laguerre polynomial of order one is given by
\begin{align}\label{quasi-type kernel Laguerre}
	L^{QC}_n(x;0)=\mathcal{L}^{(\alpha+1)}_{n}(x)+\gamma_n\mathcal{L}^{(\alpha+1)}_{n-1}(x).
\end{align}

\subsubsection{\underline{Orthogonality of quasi-Christoffel Laguerre polynomials, $L^{QC}_n(x;0)$}} \label{Sol_QCLP} To ensure the orthogonality of $L^{QC}_n(x;0)$, the condition \eqref{gamma n restriction cond} must be satisfied, which gives
\begin{align}\label{Laguerre gamma n restriction cond}
	(2+\gamma_j-\gamma_{j+1})+\frac{1}{\gamma_{j-1}}(j-1)(j+\alpha)-\frac{1}{\gamma_j}j(j+\alpha+1)=0.
\end{align}
Taking sum over $j=2$ to $n+1$, the equation is expressed as follows:
\begin{align}
	(2n+\gamma_2-\gamma_{n+2})+\frac{1}{\gamma_{1}}(\alpha+2)-\frac{1}{\gamma_{n+1}}{(n+1)}(n+\alpha+2)=0.
\end{align}
We provide two possible solutions to \eqref{Laguerre gamma n restriction cond}, which are as follows:

\textbf{Solution 1.}
By choosing $\gamma_{1}=\alpha+2$ and $\gamma_2=\alpha+3$, we recursively determine $\gamma_{n}=n+\alpha+1$. As a result, the polynomial $L^{QC}_n(x;0)$ is given by
\begin{align}\label{quasi-type kernel Laguerre orthogonal polynomial}
	L^{QC}_n(x;0)=\mathcal{L}^{(\alpha+1)}_{n}(x)+(n+\alpha+1)\mathcal{L}^{(\alpha+1)}_{n-1}(x),
\end{align}
which satisfies the three-term recurrence relation
\begin{align*}
	L^{QC}_{n+1}(x;0)=(x-c^{qc}_{n+1})L^{QC}_n(x;0)-\lambda^{qc}_{n+1}L^{QC}_{n-1}(x;0),
\end{align*} with recurrence coefficients given by
\begin{align}\label{recurrence coef quasi Christoffel Laguerre}
	\lambda^{qc}_{n+1}	=\frac{\gamma_{n}}{\gamma_{n-1}}\lambda_{n}^c=(n-1)(n+\alpha+1),~ c^{qc}_{n+1}=c_{n+1}^c+\gamma_{n}-\gamma_{n+1}=2n+\alpha+1.
\end{align}
If we put $\alpha=-1$ into equation \eqref{quasi-type kernel Laguerre orthogonal polynomial}, then the polynomial $L^{QC}_n(x;0)$ coincides with the Laguerre polynomial of degree $n$ with parameter $\alpha=-1$, i.e.,
\begin{align}
	\mathcal{L}^{(-1)}_{n}(x)=\mathcal{L}^{(0)}_{n}(x)+n\mathcal{L}^{(0)}_{n-1}(x).
\end{align}
By setting $\gamma_n=n+\alpha+1$, we ensure the orthogonality of the monic polynomial  $L^{QC}_n(x;0)$. Consequently, the constant term in the polynomial $L^{QC}_{n}(x;0)$ vanishes, implying that $p(x)=x$ is a factor of the polynomial $L^{QC}_{n}(x;0)$ for each degree $n\geq1$. We can also combine the right side of \eqref{quasi-type kernel Laguerre orthogonal polynomial} to obtain the compact form of $L^{QC}_{n}(x;0)$ in terms of the Laguerre polynomials with different parameter. Hence, \eqref{quasi-type kernel Laguerre orthogonal polynomial} can be written as:
\begin{align}\label{compactform_QCLOP_Sol1}
	L^{QC}_n(x;0)=x\mathcal{L}^{(\alpha+2)}_{n-1}(x)=\mathcal{L}^{(\alpha+1)}_{n}(x)+(n+\alpha+1)\mathcal{L}^{(\alpha+1)}_{n-1}(x).
\end{align}
This tells that the polynomials $L^{QC}_n(x;0)$ for $n\geq1$ defined in \eqref{quasi-type kernel Laguerre} with $\gamma_n=n+\alpha+1$ becomes orthogonal with respect to the measure $d\mu=x^{\alpha}e^{-x}$ for $\alpha>-1$.

\begin{remark}
Note that this connection formula can also be derived by initially applying the Christoffel transformation to Laguerre polynomials $\mathcal{L}_n^{\alpha}, \alpha > -1$, followed by the Geronimus transformation on the Christoffel-transformed Laguerre polynomials. For more details, we refer \cite[page 59]{Paco Book_2021}.
\end{remark}

\begin{remark}
	The recurrence coefficients $\lambda^{qc}_{n+1}$ and $c^{qc}_{n+1}$ defined in \eqref{recurrence coef quasi Christoffel Laguerre}, which are essential for ensuring the orthogonality of $L^{QC}_n(x;0)$ given in \eqref{quasi-type kernel Laguerre orthogonal polynomial}, are related using the recurrence parameters of the source Laguerre polynomials as outlined in \eqref{Laguerre_TTRR}. The obtained relation is given by
	\begin{align*}
		\lambda_{n+1}=\lambda^{qc}_{n+1}+c^{qc}_{n+1}.
	\end{align*}
\end{remark}

\textbf{Solution 2.} For  $\alpha>-1$, another solution to equation \eqref{Laguerre gamma n restriction cond}  is $\gamma_n=n$. Thus, the sequence ${L^{QC}_n(x;0)}$ for $n\geq0$ forms an orthogonal polynomial sequence. With $\gamma_n=n$, we recover the Laguerre polynomial with parameter $\alpha$. Specifically,
\begin{align}\label{Laguerre interms quasiChrstoffel}
	L^{QC}_n(x;0):=\mathcal{L}^{(\alpha)}_{n}(x)=\mathcal{L}^{(\alpha+1)}_{n}(x)+n\mathcal{L}^{(\alpha+1)}_{n-1}(x),
\end{align}
which becomes an orthogonal polynomial with respect to the measure $d\mu=x^{\alpha}e^{-x}dx$ and the recurrence parameters are given by
\begin{align*}
	\lambda^{qc}_{n+1}	=n(n+\alpha),~ c^{qc}_{n+1}=2n+\alpha+1.
\end{align*}
\begin{remark}
	In equation \eqref{Laguerre interms quasiChrstoffel}, it is observed that $L^{QC}_n(x;0)$ provides a transition into Laguerre polynomials with parameter  $\alpha$. This similarity corresponds to the decomposition of Laguerre polynomials as depicted in \cite[page 102]{Szego}.
\end{remark}
\begin{remark}
	It is noted that there is only one finite point at the endpoint of the interval of orthogonality for Laguerre polynomials. As a result, only two possible solutions of the non-linear difference equation \eqref{Laguerre gamma n restriction cond} are obtained, leading to the compact structure of the  polynomial $L^{QC}_n(x;0)$ defined in \eqref{quasi-type kernel Laguerre} and their corresponding measures. These polynomials are listed in Subsection \ref{Sol_QCLP} as solutions 1 and 2. Zhedanov's result \cite[Proposition 1]{Zhedanov_RST_OP_JCAM_1997} indicates that any linear spectral transformation can be written as a superposition of Christoffel and Geronimus transformations. Hence, it is expected that any other solution of \eqref{Laguerre gamma n restriction cond} also leads to a measure, which is the superposition of Christoffel and Geronimus transformations of the measures obtained in the two solutions given in this subsection.
\end{remark}

\subsubsection{\underline{Zeros of quasi-Christoffel Laguerre polynomials}}
Orthogonality fails for the polynomial $L^{QC}_n(x;0)$ at the first instance. This results in at most one zero of the polynomial $L^{QC}_{n}(x;0)$ defined in \eqref{quasi-type kernel Laguerre} lying outside the support of the measure for the Laguerre polynomials. Table \ref{Zeros_Quasi Chrsitoffel Laguerre} presents the behavior of zeros of $L^{QC}_{n}(x;0)$. For $\alpha=0$ and $\gamma_{n}=7$, we observe that one zero $(x_0=-0.404714)$ of $L^{QC}_{n}(x;0)$ lies outside the support of the measure for the  Laguerre polynomials, while all other zeros are positive. Similarly, for $\alpha=1.5$ and $\gamma_{n}=9$, Table \ref{Zeros_Quasi Chrsitoffel Laguerre} shows that at most one negative zero of the polynomial $L^{QC}_{n}(x;0)$ exists.

\begin{table}[ht]
	\begin{center}
		\resizebox{!}{2.0cm}{\begin{tabular}{|c|c|}
				\hline
				\multicolumn{2}{|c|}{Zeros of $L^{QC}_{n}(x;0)$}\\
				\hline
				$n=5$, $\alpha=0$, $\gamma_{n}=7$ &$n=6$, $\alpha=1.5$, $\gamma_{n}=9$\\
				\hline
				-0.407194&-0.219116\\
				\hline
				1.08691&1.67954\\
				\hline
				3.2637&3.90364\\
				\hline
				6.75121&7.07314\\
				\hline
				12.3054&11.5115\\
				\hline
				-&18.0513\\
				\hline
		\end{tabular}}
		\captionof{table}{Zeros of $L^{QC}_{n}(x;0)$}
		\label{Zeros_Quasi Chrsitoffel Laguerre}
	\end{center}
\end{table}

 After achieving the orthogonality of $L^{QC}_{n}(x;0)$ with $\gamma_n$ given in solution 1 of Subsection \eqref{Sol_QCLP}, we observe that one zero of $L^{QC}_{n}(x;0)$, as given in \eqref{compactform_QCLOP_Sol1}, lies on the boundary of the support of the measure. This is illustrated in Table \ref{Interlacing of quasi-Christoffel orthogonal Laguerre}, while all other zeros lie within the interval $(0,\infty)$.
  In Figure \ref{Interlace_Zeros_quasi-Christoffel orthogonal Laguerre}, for $\alpha=-0.5$, we observe the interlacing of zeros between $L^{QC}_5(x;0)$ and $L^{QC}_6(x;0)$. Additionally, Figure \ref{Interlace_Zeros_Laguerre_quasi-Christoffel orthogonal Laguerre} illustrates the interlacing between the zeros of $L^{QC}_5(x;0)$ and $\mathcal{L}^{(-0.5)}_5(x)$. Similarly, interlacing between the Christoffel transformed Laguerre  polynomial and $L^{QC}_n(x;0)$ is also demonstrated in Figure \ref{Interlace_Zeros_Christoffel Laguerre_quasi-Christoffel orthogonal Laguerre}.
  \vspace{0.1cm}

\noindent\begin{minipage}[c]{0.5\textwidth}
 \resizebox{!}{1.4cm}{\begin{tabular}{|c|c|}
     \hline  \multicolumn{2}{|c|}{Interlacing of $L^{QC}_{n}(x;0)$ and $L^{QC}_{n+1}(x;0)$}\\
				\hline
				$\alpha=-0.5$, $n=5$,
				 $\gamma_{n}=n+\alpha+1$  &$\alpha=-0.5$, $n=6$,  $\gamma_{n}=n+\alpha+1$\\
				\hline
				0.0&0.0\\
				\hline
				0.978507&0.817632\\
				\hline
				2.99038&2.47233\\
				\hline
				6.3193&5.11601\\
				\hline
				11.7118&9.04415\\
				\hline
				\hyp{}&15.0499\\
				\hline
		\end{tabular}}
	\captionsetup{type=table}
		\captionof{table}{Zeros of $L^{QC}_{n}(x;0)$}
		\label{Interlacing of quasi-Christoffel orthogonal Laguerre}
	\end{minipage}
\begin{minipage}[c]{0.5\textwidth}
	\hspace{0.7cm}\resizebox{!}{1.4cm}{\begin{tabular}{|c|c|c|}
			\hline
			\multicolumn{1}{|c|}{Zeros of $\mathcal{L}^{\alpha}_{n}(x)$}&\multicolumn{1}{|c|}{Zeros of $\mathcal{C}_{n}(x;0)$}&\multicolumn{1}{|c|}{Zeros of $L^{QC}_n(x;0)$}\\
			\hline
			$\alpha=-0.5$, $n=5$  &$\alpha=-0.5$, $n=5$ & $\alpha=2$, $n=5$\\
			\hline
			0.117581&0.431399&0.0\\
			\hline
			1.07456&1.75975&2.31916\\
			\hline
			3.08594&4.10447&5.12867\\
			\hline
			6.41473&7.7467&9.20089\\
			\hline
			11.8072&13.4577&15.3513\\
			\hline
			-&-&-\\
			\hline
				\end{tabular}}
\captionsetup{type=table}
\captionof{table}{Zeros of $\mathcal{L}^{\alpha}_n(x)$ }
\label{Zeros of quasi-Christoffel orthogonal Laguerre}
\end{minipage}
\begin{figure}[!ht]
	\includegraphics[scale=0.9]{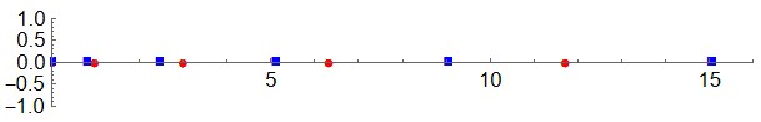}
	\caption{Zeros of $L^{QC}_{5}(x;0)$ (blue squares) and $L^{QC}_{6}(x;0)$ (red circles).}
	\label{Interlace_Zeros_quasi-Christoffel orthogonal Laguerre}
\end{figure}
\begin{figure}[!ht]
	\includegraphics[scale=0.9]{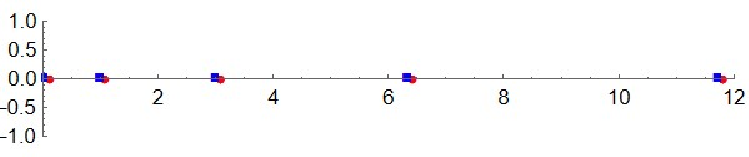}
	\caption{Zeros of $L^{QC}_{5}(x;0)$ (blue squares) and $\mathcal{L}^{(-0.5)}_{5}(x)$ (red circles).}
	\label{Interlace_Zeros_Laguerre_quasi-Christoffel orthogonal Laguerre}
\end{figure}
\begin{figure}[!ht]
	\includegraphics[scale=0.9]{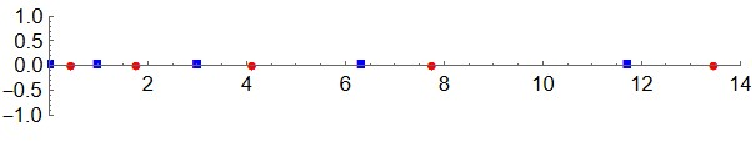}
	\caption{Zeros of $L^{QC}_{5}(x;0)$ (blue squares) and $\mathcal{C}_{5}(x;0)$ (red circles).}
	\label{Interlace_Zeros_Christoffel Laguerre_quasi-Christoffel orthogonal Laguerre}
\end{figure}

It is worth noting that the sequence of Laguerre polynomials $\{\mathcal{L}_n^{(\alpha)}(x)\}_{n=0}^{\infty}$ becomes classically orthogonal when $\alpha>-1$. However, substituting $\alpha=-1$ breaks this classical orthogonality condition, necessitating orthogonality in the non-classical sense, such as Sobolev orthogonality.  The tail-end sequence of Laguerre polynomials $\{\mathcal{L}_n^{(-1)}(x)\}_{n=1}^{\infty}$ becomes orthogonal with respect to the usual inner product. The more general framework of the orthogonality of sequence of Laguerre polynomials $\{\mathcal{L}_n^{(\alpha)}(x)\}_{n=\alpha}^{\infty}, \alpha=-m, m\in \mathbb{N}$ is discussed in \cite{Everitt_Littlejohn_Wellman_Sobolev orth_Laguerre_JCAM_2004}.  Further exploration of Laguerre polynomial orthogonality in the non-classical sense can be found in \cite{Everitt_Littlejohn_Wellman_Sobolev orth_Laguerre_JCAM_2004,Hajmirzaahmad_Laguerre for alpha -1}.

To extend the applicability of Laguerre polynomials to negative integral values of $\alpha$, i.e., $\alpha=-m, m\in \mathbb{N}$, we can utilize the following formula (see \cite[equation (5.2.1)]{Szego}):
\begin{align}\label{Laguerre neg-pos parameter relation}
	\mathcal{L}_n^{(-m)}(x)=(-1)^mx^m\frac{\Gamma(n-m+1)}{\Gamma(n+1)}\mathcal{L}_{n-m}^{(m)}(x), ~~~\text{for}~~m\leq n.
\end{align}
When $m=1$ in \eqref{Laguerre neg-pos parameter relation}, we observe that $x=0$ becomes a common zero of the polynomial $\mathcal{L}_n^{(-1)}(x)$ for $n\geq1$. The polynomial described in equation \eqref{quasi-type kernel Laguerre orthogonal polynomial}, which is obtained by achieving orthogonality of $L^{QC}_n(x;0)$, has a common zero at $x=0$. This zero remains consistent regardless of the values of the parameter $\alpha$, as also shown in Table \ref{Interlacing of quasi-Christoffel orthogonal Laguerre} and Table \ref{Zeros of quasi-Christoffel orthogonal Laguerre}.

\begin{remark}	The derivative of a Laguerre polynomial yields another Laguerre polynomial \cite[page 149]{Chihara book}. Specifically,
	\begin{align}\label{Derivative Laguerre}
		\frac{d\mathcal{L}_n^{\alpha}(x)}{dx}=n\mathcal{L}_{n-1}^{\alpha+1}(x).
	\end{align}
 Thus, it follows that a linear combination of the Christoffel Laguerre polynomial and the derivative of the Laguerre polynomial constitutes an orthogonal polynomial. This can be achieved by substituting \eqref{Derivative Laguerre} into \eqref{quasi-type kernel Laguerre orthogonal polynomial}.
\end{remark}

\begin{remark}We have illustrated a behaviour of the zeros of the polynomial $L^{QC}_n(x;0)$ obtained in solution 1 of Subsection \ref{Sol_QCLP}. However, the polynomial obtained in solution 2 of Subsection \ref{Sol_QCLP} is again a family of the Laguerre orthogonal polynomials, whose numerical illustrations can be found in the literature (see \cite{Driver_Littlejohn_zeros_Laguerre-2021, Driver_Muldoon_zeros-Laguerre_JAT_2015}). Hence, we have not illustrated the zeros of the polynomials obtained in solution 2, explicitly.
\end{remark}
\subsection{Chain sequences and quasi-Christoffel polynomials}

The connection between the Chain sequences and orthogonal polynomials is well known. Specifically, the chain sequences enable us to establish a relationship between the support of measure  and the recurrence coefficients. A sequence $\{s_n\}_{n=1}^{\infty}$ is defined as a chain sequence if there exists a parameter sequence $\{d_n\}_{n=0}^{\infty}$ such that
\begin{align}
	s_n=(1-d_{n-1})d_n,
\end{align}
where $d_0\in [0,1)$ and $d_n\in(0,1)$ for $n\geq 1$.  If $[a,b]$ represents the support of the measure for the orthogonal polynomial $\mathbb{P}_n(x)$ and $t\leq a$, then $\{s_n(t)\}$ for $n=1,2,...$, defined by
\begin{align}\label{chain sequence using recurr para}
	s_n(t)=\frac{\lambda_{n+1}}{(c_n-t)(c_{n+1}-t)},
\end{align}
forms a chain sequence. The expression of the minimal parameter sequence in terms of orthogonal polynomials is also well known. If $t\not\in (a,b)$, then the chain sequence can be expressed as
\begin{align}
	s_n(t)=(1-m_{n-1}(t))m_n(t),
\end{align}
where the parameter chain sequence $\{m_n(t)\}$ for $n=0,1,2,...$ is given by
\begin{align}\label{minimal parameter sequence}
	\nonumber m_0(t)&=0,\\
	m_n(t)&=1-\frac{\mathbb{P}_{n+1}(t)}{(t-c_{n+1})\mathbb{P}_n(t)}, ~~\text{for}~~n\in \mathbb{N}.
\end{align}
For more detailed information about the chain sequence, we refer to \cite{Chihara book}.

The recurrence parameters involved in the Jacobi matrix play a crucial role in the study of chain sequences. Consequently, we discuss the chain sequence corresponding to the quasi-Christoffel polynomials that satisfy \eqref{TTRR_QCP}. Thus, if the polynomials $\mathcal{C}^{Q}_n(x;a)$ are orthogonal with respect to $d\nu$ and $\text{supp}(d\nu)$ represents the support of the measure $d\nu$, then for  $x\leq \text{inf~supp}(d\nu)$, the chain sequence $\tilde{s}_n(x)$ for $n=2,3,..$ is given by
\begin{align}\label{chain sequence using recurr para_QCOP}
	\tilde{s}_n(x)	=\frac{\lambda^{qc}_{n+1}}{(c^{qc}_n-x)(c^{qc}_{n+1}-x)}=\frac{\gamma_n\lambda_{n}^c}{\gamma_{n-1}(c_{n+1}^c+\gamma_{n}-\gamma_{n+1}-x)(c_{n}^c+\gamma_{n-1}-\gamma_{n}-x)},
\end{align}
where the parameter chain sequence $\{\tilde{m}_n(x)\}$ for $n=1,2,...$ is given by
\begin{align}\label{minimal parameter sequence_QCOP}
	\nonumber	\tilde{m}_1(x)&=0,\\
	\tilde{m}_n(x)&=1-\frac{\mathcal{C}^{Q}_{n+1}(x;a)}{(x-c_{n+1}^{qc})\mathcal{C}^{Q}_n(x;a)}, ~~\text{for}~~n\geq2.
\end{align}

\begin{example}[Chain sequence for quasi-Christoffel Laguerre polynomial]
	We observe that exactly one zero of $L^{QC}_n(x;0)$ defined in \eqref{quasi-type kernel Laguerre orthogonal polynomial} lies at the finite end point  of the support of the measure for the Laguerre polynomial. Nevertheless, we can still derive the chain sequence and minimal parameter sequence. This is made possible by the cancellation of the common zero in \eqref{minimal parameter sequence_QCOP}. Utilizing the recurrence parameters of $L^{QC}_n(x;0)$, as defined in \eqref{recurrence coef quasi Christoffel Laguerre}, we derive the chain sequence $\{\tilde{s}_n(x)\}_{n=2}^{\infty}$  at $x=0$. The corresponding chain sequence is given by
	\begin{align*}
		\tilde{s}_n(0)=\frac{(n-1)(n+\alpha+1)}{(2n+\alpha+1)(2n+\alpha-1)}=(1-m_{n-1}(0))m_n(0),
	\end{align*}
	where the parameter sequence $\{\tilde{m}_n(0)\}_{n=1}^{\infty}$ is given by
	\begin{align*}
		\tilde{m}_n(0)=\frac{n-1}{2n+\alpha+1}.
	\end{align*}
	In fact, the parameter sequence $\tilde{m}_n(0)$ is a minimal parameter sequence, as $\tilde{m}_1(0)=0$. It is evident that the minimal parameter sequence $\tilde{m}_n(0)>0$ for $n\geq2$ and $\alpha>-1$. Additionally, the strict upper bound for $\tilde{m}_n(0)$ is $\frac{1}{2}$ for $n\geq2$ and $\alpha>-1$. This can be shown as follows:
	\begin{align*}
		\tilde{m}_n(0)=\frac{n-1}{2n+\alpha+1}\leq \frac{n-1}{2n}=\frac{1}{2}\left(1-\frac{1}{n}\right)<\frac{1}{2}.
	\end{align*}
	Thus, according to \cite[Lemma 2.5]{Behera_Ranga_Swami_SIGMA_2016}, the complementary chain sequence  of the chain sequence $\tilde{s}_n(0)$ is SPPCS.
\end{example}
In a similar manner, we can derive the chain sequence and parameter sequence for the quasi-Christoffel Jacobi polynomial of order one for the various solutions discussed in  Subsection \ref{Sol1-QCJP}.

\section{Quasi-Geronimus polynomial of order one}\label{sec:Quasi-Geronimus}
In this section, we discuss the solutions of the non-linear difference equation that provides the orthogonality conditions for quasi-Geronimus Jacobi polynomials of order one, along with the explicit expression of the corresponding orthogonal polynomials.

 Let us define the linear functional $\mathcal{L}^G$ in terms of the linear functional $\mathcal{L}$ as:
\begin{align*}
	\mathcal{L}^G(p(x))=\mathcal{L}\left(\frac{p(x)-p(a)}{x-a}\right)+N p(a),
\end{align*}
for any polynomial $p$ and constant $N$. This perturbed linear functional $\mathcal{L}^G$ at point $a$ is known as the canonical Geronimus transformation. If $N\neq0$ and $a$ does not belong to the support of the measure for $\mathbb{P}_n$, then there exists a sequence of orthogonal polynomials known as Geronimus polynomials with respect to the linear functional $\mathcal{L}^G$.  The explicit expression of the Geronimus polynomial can be given as:
\begin{align}\label{Geronimus polynomial}
	\mathcal{G}_n(x;a)=\mathbb{P}_n(x)+t_n(a)\mathbb{P}_{n-1}(x), ~n\geq1,
\end{align}
where
\begin{align}
	t_n(a)=-\frac{\mathcal{L}(1)\mathbb{P}^{(1)}_{n-1}(a)+N\mathbb{P}_n(a)}{\mathcal{L}(1)\mathbb{P}^{(1)}_{n-2}(a)+N\mathbb{P}_{n-1}(a)}, ~n\geq 1,
\end{align}
and  the polynomial  $\mathbb{P}^{(1)}_{n}$ of degree $n-1$ is known as either associated polynomial of the first kind or numerator polynomial (see \cite{Chihara book}). The Geronimus polynomial also satisfies the TTRR
\begin{align}\label{Geronimus-TTRR}x\mathcal{G}_n(x;a)=\mathcal{G}_{n+1}(x;a)+c_{n+1}^g\mathcal{G}_n(x;a)+ \lambda_{n+1}^g\mathcal{G}_{n-1}(x;a), ~ n\geq0,
\end{align}
with the recurrence parameter
\begin{align*}
	c_{n+1}^g=c_{n+1}-t_{n}(a)+t_{n+1}(a), n\geq0 ,~~~~\lambda_{n+1}^g=\lambda_{n}\frac{t_{n}(a)}{t_{n-1}(a)}, ~ n\geq1.
\end{align*}
The quasi-orthogonality between two consecutive degrees of Geronimus polynomials, known as the quasi-Geronimus polynomial of order one, is explored in \cite{Vikas_Paco_Swami_quasi-nature}. The next result presents the characterization of the quasi-Geronimus polynomial of order one.

\begin{theorem}\rm\cite{Vikas_Paco_Swami_quasi-nature}\label{quasi-Geronimus}
Let $\mathcal{G}_n(x;a)$ denote the monic polynomial associated with the canonical Geronimus transformation at a certain point $a$. The monic polynomial $\mathcal{G}_{n+1}^Q(x;a)$ of degree $n+1$ is a non-trivial quasi-Geronimus polynomial of order one with respect to $\mathcal{L}^G$ if and only if there exists a sequence of non-zero constants $\chi_n$, such that:
	\begin{align}\label{quasi-Geronimus order1_characterization}
		\mathcal{G}_{n+1}^Q(x;a)= \mathcal{G}_{n+1}(x;a) + \chi_{n+1}\mathcal{G}_{n}(x;a).
	\end{align}
\end{theorem}
In the next result, we present the orthogonality of quasi-Geronimus polynomial of order one under certain assumptions on $\chi_{n}$, as discussed in \cite[Propostion 1]{Vikas_Paco_Swami_quasi-nature}.
\begin{proposition}\label{orthogonality of quasi-Geronimus}
	Let $\mathcal{G}_{n}(x;a)$ be a monic Geronimus polynomial of degree $n$ with respect to the linear functional $\mathcal{L}^G$ at point $a$. Suppose also that $\mathcal{G}_{n}^{Q}(x;a)$ given in \eqref{quasi-Geronimus order1_characterization}, is a monic quasi-Geronimus polynomial of order one with parameter $\chi_{n}$, which satisfies the following non-linear difference equation:
	\begin{align}\label{beta n restriction cond}
		c_{n+1}^g-c_n^g+\chi_{n}-\chi_{n+1}+\frac{\lambda_{n}^g}{\chi_{n-1}}-\frac{\lambda_{n+1}^g}{\chi_{n}}=0, ~n\geq2.
	\end{align}
	With these assumptions, the polynomials $\mathcal{G}_{n}^{Q}(x;a)$ satisfy the TTRR
	\begin{align*}
		\mathcal{G}_{n+1}^{Q}(x;a)-(x-c_{n+1}^{qg})\mathcal{G}_{n}^{Q}(x;a)+\lambda_{n+1}^{qg}\mathcal{G}_{n-1}^{Q}(x;a)=0, ~ n\geq 0,
	\end{align*}
	where the recurrence parameters are given by
	\begin{align*}
		\lambda_{n+1}^{qg}=\frac{\chi_{n}}{\chi_{n-1}}\lambda_{n}^g, ~~~~~~~c_{n+1}^{qg}=c_{n+1}^g+\chi_{n}-\chi_{n+1}.
	\end{align*}
	If $\lambda_{n+1}^{qg}\neq 0$, then  the sequence $\{\mathcal{G}^Q_n(x;a)\}_{n=1}^{\infty}$ becomes orthogonal with respect to a quasi-definite linear functional.
\end{proposition}

\subsection{Quasi-Geronimus Jacobi polynomials of order one}
When applying the canonical Geronimus transformation to the Jacobi weight, $w(x)=(1-x)^{\alpha}(1+x)^{\beta}$, with $\alpha>-1$ and $\beta>-1$, using $a=-1$ and $N=2^{\alpha+\beta}\mathbb{B}(\alpha+1,\beta)$, where $\mathbb{B}(\cdot,\cdot)$ denotes the Beta function, the resulting transformed weight becomes $\tilde{w}(x)=(1-x)^{\alpha}(1+x)^{\beta-1}$ for $\alpha>-1$ and $\beta>0$. Therefore, the orthogonal polynomial corresponding to the Geronimus transformed Jacobi weight is again the Jacobi polynomial with parameter $(\alpha,\beta-1)$. Denoted by $\mathcal{G}_n(x;-1):=\mathcal{P}_n^{(\alpha,\beta-1)}(x)$, this polynomial can be obtained using a recurrence relation with recurrence parameters:
 \begin{align*}
	\lambda^g_{n+1}&=\frac{4n(n+\alpha)(n+\beta-1)(n+\alpha+\beta-1)}{(2n+\alpha+\beta-1)^2(2n+\alpha+\beta)(2n+\alpha+\beta-2)},\\
	c^g_{n+1}&=\frac{(\beta-1)^2-\alpha^2}{(2n+\alpha+\beta-1)(2n+\alpha+\beta+1)}.
\end{align*}
The quasi-Geronimus Jacobi polynomial of order one is defined as:
\begin{align}\label{quasi-Geronimus Jacobi poly.}
	P^{QG}_{n}(x;-1)=\mathcal{P}^{(\alpha,\beta-1)}_{n}(x)+\chi_{n}\mathcal{P}^{(\alpha,\beta-1)}_{n-1}(x).
\end{align}

 \subsubsection{\underline{Orthogonality of quasi-Geronimus Jacobi polynomials}}\label{Sol_OGJP}

The polynomial defined in \eqref{quasi-Geronimus Jacobi poly.} becomes orthogonal if $\chi_{n}$ satisfies the following non-linear difference equation:
\begin{align}\label{chi n condition for Jacobi}
	\nonumber	&-\frac{(\beta-1)^2-\alpha^2}{(\alpha+\beta+2 n-3) (\alpha+\beta+2 n-1)}+\frac{(\beta-1)^2-\alpha^2}{(\alpha+\beta+2 n-1) (\alpha+\beta+2 n+1)}+\chi_n-\chi_{n+1}\\
	\nonumber&\hspace{1cm}	-\frac{1}{\chi_{n}}\frac{4 n (\alpha+n) (\beta+n-1) (\alpha+\beta+n-1)}{ (\alpha+\beta+2 n-2) (\alpha+\beta+2
		n-1)^2 (\alpha+\beta +2n)}\\
	&\hspace{3cm}+\frac{1}{\chi_{n-1}}\frac{4 (n-1) (\alpha+n-1) (\beta+n-2) (\alpha+\beta+n-2)}{(\alpha+\beta+2 n-4) (\alpha+\beta+2 n-3)^2 (\alpha+\beta+2 n-2)}=0.
\end{align}
We provide four possible solutions to \eqref{chi n condition for Jacobi}, which are as follows:

\textbf{Solution 1.} It is easy to see that
\begin{align}\label{Explicit chi value for QGJP_Sol1}
	\chi_{n}=-\frac{2(\alpha+n)(n+\alpha+\beta-1)}{(2n+\alpha+\beta-1)(2n+\alpha+\beta-2)},
\end{align}
solves the non-linear difference equation \eqref{chi n condition for Jacobi}.Thus, the polynomial $P^{QG}_{n}(x;-1)$  becomes orthogonal with the $\chi_{n}$ defined in \eqref{Explicit chi value for QGJP_Sol1}. Moreover, the recurrence parameters for obtaining $P^{QG}_{n}(x;-1)$ are given by:
\begin{align}\label{rec._Coeff_quasi_geronimus_Orth._Jacobi}
\nonumber	\lambda^{qg}_{n+1}&=\frac{4(n-1)(n+\alpha)(n+\beta-2)(n+\alpha+\beta-1)}{(2n+\alpha+\beta-2)^2(2n+\alpha+\beta-1)(2n+\alpha+\beta-3)},\\
	c^{qg}_{n+1}&=\frac{(\beta-\alpha-2)(\beta+\alpha)}{(2n+\alpha+\beta-2)(2n+\alpha+\beta)},
\end{align}
and the compact form of the polynomial $P^{QG}_{n}(x;-1)$ given in \eqref{quasi-Geronimus Jacobi poly.} with coefficient $\chi_n$ in \eqref{Explicit chi value for QGJP_Sol1} can be written as:

$\displaystyle P^{QG}_{n}(x;-1)=(x-1)\mathcal{P}^{(\alpha+1,\beta-1)}_{n-1}(x)$
\begin{align}\label{Compactform_QGJOP}
	\hspace{-1.2cm}=\mathcal{P}^{(\alpha,\beta-1)}_{n}(x)-\frac{2(\alpha+n)(n+\alpha+\beta-1)}{(2n+\alpha+\beta-1)(2n+\alpha+\beta-2)}\mathcal{P}^{(\alpha,\beta-1)}_{n-1}(x).
\end{align}
The polynomial $P^{QG}_{n}(x;-1)$ for $n\geq2$ defined in \eqref{quasi-Geronimus Jacobi poly.} with $\chi_{n}$ in \eqref{Explicit chi value for QGJP_Sol1} is orthogonal with respect to the measure $d\mu(x)=(1-x)^{\alpha-1}(1+x)^{\beta-1}dx, \alpha>-1, \beta>-1$.
Also, the monic Jacobi matrix with recurrence parameters $\lambda^{qc}_{n+1}$ and $c^{qc}_{n+1}$ given by \eqref{rec._Coeff_quasi_geronimus_Orth._Jacobi} is also bounded.  The measure corresponding to the orthogonal polynomial defined by \eqref{quasi-Geronimus Jacobi poly.} with $\chi_{n}$ specified in \eqref{Explicit chi value for QGJP_Sol1} belongs to the Nevai class $\mathbb{N}(1/4,0)$. This can be obtained by an argument similar to that in Subsection \ref{sub:QCJP}.


\textbf{Solution 2.} To recover the orthogonality of the polynomial $P^{QG}_{n}(x;-1)$ defined in \eqref{quasi-Geronimus Jacobi poly.}, we find that:
\begin{align}\label{Geronimus_nonlinear_diffeqn_Sol2}
	\chi_{n}=\frac{2n(\alpha+n)}{(2n+\alpha+\beta-1)(2n+\alpha+\beta-2)},
\end{align}
also satisfies the equation \eqref{chi n condition for Jacobi}. With this $\chi_n$, the polynomial $P^{QG}_{n}(x;-1)$ becomes orthogonal and the compact form of the polynomial is given by:
\begin{align}
	P^{QG}_{n}(x;-1):=\mathcal{P}^{(\alpha,\beta-2)}_{n}(x)=\mathcal{P}^{(\alpha,\beta-1)}_{n}(x)+\frac{2n(\alpha+n)}{(2n+\alpha+\beta-1)(2n+\alpha+\beta-2)}\mathcal{P}^{(\alpha,\beta-1)}_{n-1}(x).
\end{align}
The recurrence parameter are given by:
\begin{align}\label{recurrence coef. quasi Geron Jacobi_Sol2}
	\nonumber	\lambda^{qg}_{n+1}&=\frac{4n(n+\alpha)(n+\beta-2)(n+\alpha+\beta-2)}{(2n+\alpha+\beta-2)^2(2n+\alpha+\beta-1)(2n+\alpha+\beta-3)},\\
	c^{qg}_{n+1}&=\frac{(\beta-2)^2-\alpha^2}{(2n+\alpha+\beta)(2n+\alpha+\beta-2)}.
\end{align}
The orthogonality measure of the polynomials $P^{QG}_{n}(x;-1)$ for $n\geq2$ defined in \eqref{quasi-Geronimus Jacobi poly.} with $\chi_{n}$ in \eqref{Geronimus_nonlinear_diffeqn_Sol2} is given by $d\mu(x)=(1-x)^{\alpha}(1+x)^{\beta-2}dx, \alpha>-1, \beta>-1$.

\textbf{Solution 3.} Another solution of the nonlinear difference equation  \eqref{chi n condition for Jacobi} is given by:
\begin{align}\label{Geronimus_nonlinear_diffeqn_Sol3}
	\chi_{n} = -\frac{2n(\beta+n-1)}{(2n+\alpha+\beta-1)(2n+\alpha+\beta-2)}.
\end{align}
Thus the polynomial $P^{QG}_{n}(x;-1)$ defined in \eqref{quasi-Geronimus Jacobi poly.} with $\chi_n$ in \eqref{Geronimus_nonlinear_diffeqn_Sol3} can be written as:
{\small \begin{align}\label{QCGOP_Sol3}
	P^{QG}_{n}(x;-1) := \mathcal{P}^{(\alpha-1,\beta-1)}_{n}(x) = \mathcal{P}^{(\alpha,\beta-1)}_{n}(x)-\frac{2n(\beta+n-1)}{(2n+\alpha+\beta-1)(2n+\alpha+\beta-2)} \mathcal{P}^{(\alpha,\beta-1)}_{n-1}(x),
\end{align}}
and the recurrence parameters are given by
\begin{align}\label{recurrence coef. quasi Geron Jacobi_Sol3}
	\nonumber	\lambda^{qg}_{n+1}&=\frac{4n(n+\alpha-1)(n+\beta-1)(n+\alpha+\beta-2)}{(2n+\alpha+\beta-2)^2(2n+\alpha+\beta-1)(2n+\alpha+\beta-3)},\\
	c^{qg}_{n+1}&=\frac{(\beta-1)^2-(\alpha-1)^2}{(2n+\alpha+\beta)(2n+\alpha+\beta-2)}.
\end{align}
The orthogonality measure of the polynomials $P^{QG}_{n}(x;-1)$ defined in \eqref{QCGOP_Sol3} for $n\geq3$ is given by $d\mu(x)=(1-x)^{\alpha-1}(1+x)^{\beta-1}dx, \alpha>-1, \beta>-1$.

\textbf{Solution 4.}  Also,
\begin{align}\label{Geronimus_Jacobi_nonlinear_difference_Sol4}
	\chi_{n}=\frac{2(\beta+n-1)(n+\alpha+\beta-1)}{(2n+\alpha+\beta-1)(2n+\alpha+\beta-2)},
\end{align}
solves \eqref{chi n condition for Jacobi}. Hence, the polynomial $P^{QG}_{n}(x;-1)$ together with $\chi_n$ in \eqref{Geronimus_Jacobi_nonlinear_difference_Sol4} becomes orthogonal and its compact form is given by::
{\small\begin{align}\label{Jacobi interms extend parameter Jacobi_Sol4}
		P^{QG}_{n}(x;-1) := (x+1)\mathcal{P}^{(\alpha,\beta)}_{n-1}(x) = \mathcal{P}^{(\alpha,\beta-1)}_{n}(x) + \frac{2(\beta+n-1)(n+\alpha+\beta-1)}{(2n+\alpha+\beta-1)(2n+\alpha+\beta-2)} \mathcal{P}^{(\alpha,\beta-1)}_{n-1}(x).
\end{align}}
The corresponding recurrence parameters are:
\begin{align}\label{rec._Coeff_quasi_geronimus_Orth._Jacobi_Sol4}
	\nonumber	\lambda^{qg}_{n+1}&=\frac{4(n-1)(n+\alpha-1)(n+\beta-1)(n+\alpha+\beta-1)}{(2n+\alpha+\beta-2)^2(2n+\alpha+\beta-1)(2n+\alpha+\beta-3)},\\
	c^{qg}_{n+1}&=\frac{\beta^2-\alpha^2}{(2n+\alpha+\beta)(2n+\alpha+\beta-2)}.,
\end{align}
The orthogonality measure of the polynomials $P^{QG}_{n}(x;-1)$ for $n\geq2$ defined in \eqref{quasi-Geronimus Jacobi poly.} with $\chi_{n}$ in \eqref{Geronimus_Jacobi_nonlinear_difference_Sol4} is given by $d\mu(x)=(1-x)^{\alpha}(1+x)^{\beta-2}dx, \alpha>-1, \beta>-1$.
\begin{remark}
	Any other solution of \eqref{chi n condition for Jacobi} gives the orthogonality of the quasi-Geronimus Jacobi polynomial \eqref{quasi-Geronimus Jacobi poly.}, whose measure is the superposition of Christoffel and Geronimus transformations of the measures obtained in solutions 1, 2, 3, and 4 discussed in Subsection \ref{Sol_OGJP}.
\end{remark}

\subsubsection{\underline{Zeros of quasi-Geronimus and Christoffel Jacobi polynomials}}

In this subsection, we demonstrate the interlacing property between the zeros of $J^{QC}_{n}(x;-1)$ and $P^{QG}_{n}(x;-1)$. For $\alpha=1$ and $\beta=0.5$, Table \ref{Zeros_QCJP_QGOP} and Figure \ref{Interlace_Zeros_QCJOP_QGJOP_n6} demonstrate that the zeros of $P^{QG}_{6}(x;-1)$ and $J^{QC}_{6}(x;-1)$ in \eqref{quasi-Christoffel Jacobi poly.} with $\gamma_n$ \eqref{Explicit gamma value for QCJP_Sol1} interlace inside the support of the measure. Similarly, interlacing occurs with parameter $\alpha=2$ and $\beta=1$, as shown in Table \ref{Zeros_QCJP_QGOP} and Figure \ref{Interlace_Zeros_QCJOP_QGJOP_n5}.

\begin{table}[H]
	\begin{center}
		\resizebox{!}{1.6cm}{\begin{tabular}{|c|c|c|c|}
				\hline
				\multicolumn{2}{|c|}{Zeros of $J^{QC}_n(x;-1)$ \eqref{quasi-Christoffel Jacobi poly.} with $\gamma_n$ in \eqref{Explicit gamma value for QCJP_Sol1}}&\multicolumn{2}{|c|}{Zeros of $P^{QG}_n(x;-1)$}\\
				\hline
				$n=6$,	$\alpha=1$, $\beta=0.5$ &$n=5$, $\alpha=2$, $\beta=1$&$n=6$, $\alpha=1$, $\beta=0.5$ &$n=5$, $\alpha=2$, $\beta=1$\\
				\hline
				-0.810015&-0.728794&-0.967813&-0.915694\\
				\hline
				-0.47303&-0.325544&-0.722321&-0.580566\\
				\hline
				-0.047073&-0.147611&-0.292037&-0.071692\\
				\hline
				0.389257&0.599035&0.216697&0.477044\\
				\hline
				0.755676&1&0.678518&1\\
				\hline
				1&-&1&-\\
				\hline
		\end{tabular}}
		\captionof{table}{Zeros of  $P^{QG}_n(x;-1)$ with $\chi_{n}$ given in \eqref{Explicit chi value for QGJP_Sol1}}
		\label{Zeros_QCJP_QGOP}
	\end{center}
\end{table}
\begin{figure}[H]
	\includegraphics[scale=0.9]{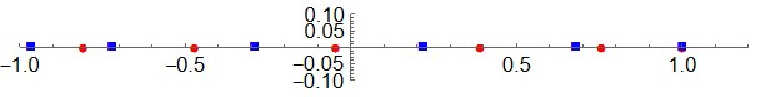}
	\caption{Zeros of $P^{QG}_{6}(x;-1)$ (blue squares) and $J^{QC}_{6}(x;-1)$ (red circles).}
	\label{Interlace_Zeros_QCJOP_QGJOP_n6}
\end{figure}
\begin{figure}[H]
	\includegraphics[scale=0.9]{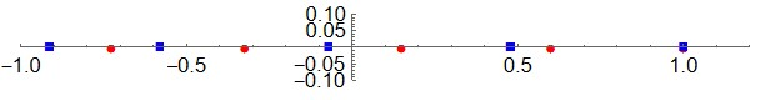}
	\caption{Zeros of $P^{QG}_{5}(x;-1)$ (blue squares) and $J^{QC}_{5}(x;-1)$ (red circles).}
	\label{Interlace_Zeros_QCJOP_QGJOP_n5}
\end{figure}

In this work, computations of the zeros and graphical representations, showcasing their interlacing properties, are conducted using the $\text{Mathematica}^{\text{\textregistered}}$ software.

\section{Quasi-nature spectral polynomials on the unit circle}\label{Sec:Conclusion}
%

Examples were considered in the previous two sections to analyse the behaviour of zeros of quasi-natured spectral polynomials that are defined on subsets of $\mathbb{R}$. To obtain further information, it is proposed to extend the study to the unit circle. We conclude this manuscript by analyzing the location of zeros of the quasi-Christoffel polynomial on the unit circle through specific examples. Additionally, we propose the problem to extend these examples to a more general framework.

 Let $d\hat{\nu}$ denote a positive Borel measure on the unit circle $\partial\mathbb{D}$, and let $\{\Phi_n\}_{n \geq 0}$ represent the sequence of monic orthogonal polynomials corresponding to $d\hat{\nu}$ on $\partial\mathbb{D}$. Starting with the initial condition $\Phi_0(z) = 1$ and a sequence of complex numbers $\{\alpha_n\}_{n\geq0}$ within the unit disc $\mathbb{D}$, one can construct the sequence $\{\Phi_n\}_{n \geq 0}$ using Szegő recursion. The constants $\alpha_n$, known as Verblunsky coefficients, are related to the polynomials $\Phi_n$ through the expression $\alpha_n = -\overline{\Phi_{n+1}(0)}$. For a comprehensive discussion of orthogonal polynomials on the unit circle, see \cite{SimonOPUC1}.

  Consider the measure $d\hat{\mu} = |z - \tilde{\gamma}|^2 d\hat{\nu}$, which is the canonical Christoffel transformation of $d\hat{\nu}$ on $\partial\mathbb{D}$, where $\tilde{\gamma} \in \mathbb{C}$ (see \cite{Garza_CT for matrix measure_CMFT_2021}). If the Christoffel-Darboux (CD) kernel, denoted by $\mathcal{K}_{n}(\tilde{\gamma}, \tilde{\gamma}, \hat{\nu})$, is strictly positive for $n \geq 0$, then there exists a sequence of monic orthogonal polynomials $\{\Phi_n(z; \tilde{\gamma})\}_{n \geq 0}$ with respect to $d\hat{\mu}$ (see \cite[Proposition 2.4]{Paco_ST for HTM_JCAM_2007}). The polynomial $\Phi_n(z; \tilde{\gamma})$ is given by
\begin{align}
	\Phi_{n-1}(z; \tilde{\gamma}) = \frac{1}{z - \tilde{\gamma}} \left( \Phi_n(z) - \frac{\Phi_n(\tilde{\gamma})}{\mathcal{K}_{n-1}(\tilde{\gamma}, \tilde{\gamma}, \nu)} \mathcal{K}_{n-1}(z, \tilde{\gamma}, \nu) \right).
\end{align}

We define the quasi-Christoffel polynomial of order one on the unit circle as
\begin{align}\label{QCP_Unit_circle}
\Phi_{n}(z;\tilde{\gamma},a_n)=\Phi_{n}(z;\tilde{\gamma})-a_n\Phi_{n-1}(z;\tilde{\gamma}).
\end{align}
Next, we analyse the location of the zeros of quasi-Christoffel polynomials on the unit circle for various choices of $\tilde{\gamma}$. Specifically, consider $d\hat{\nu}$ as the normalized Lebesgue measure and let $\tilde{\gamma} = 1$. In this case, the Christoffel-transformed measure becomes $d\hat{\mu} = |z - 1|^2 \frac{d\theta}{2\pi}$, and the corresponding Verblunsky coefficients  $\alpha_n = -\frac{1}{n+2}$ for $n \geq 0$. The monic orthogonal polynomials $\Phi_n(z; 1)$ with respect to $d\hat{\mu}$ can be expressed as:
\begin{align}
	\Phi_{n-1}(z; 1) = \frac{1}{z - 1} \left( z^n - \frac{1}{n} \sum_{k=0}^{n-1} z^k \right).
\end{align}
Thus, \eqref{QCP_Unit_circle} can be written as
\begin{align}\label{Quasi OP wrt CT measure with gamma 1}
	\Phi_{n}(z;1,a_n)
	=\frac{1}{z-1}\left(z^{n}(z-a_n)+\frac{a_n}{n}\sum_{k=0}^{n-1}z^k-\frac{1}{n+1}\sum_{k=0}^{n}z^k\right).
\end{align}
For arbitrary choices of the parameter $a_n$, the polynomial $\Phi_{n}(z; 1, a_n)$ may not satisfy orthogonality condition. Consequently, its zeros are not guaranteed to remain within the unit disc. Numerical experiments further indicate that some zeros of $\Phi_{n}(z; 1, a_n)$ can appear outside the unit disc.

Table \ref{Zeros at gamma 1} and Figure \ref{Location of zeros of QuasiOPUC for 1} illustrate that the zeros of $\Phi_n(z;1)$ are located inside the unit disc. However, for various values of $a_n$, such as $a_n = \frac{1}{n+1} - i$ and $a_n = -1.16$, at most one zero of $\Phi_n(z;1, a_n)$ lie outside the unit disc. On the other hand, when $a_n = \frac{n}{n+1}$, the quasi-Christoffel polynomial of order one \eqref{Quasi OP wrt CT measure with gamma 1} at $\tilde{\gamma} = 1$ becomes orthogonal with respect to the normalized Lebesgue measure, ensuring that all zeros lie inside the unit disc. 

\begin{table}[H]
	\begin{center}
		\resizebox{!}{1.51cm}{\begin{tabular}{|c|c|c|c|c|}
				\hline
				\multicolumn{2}{|c|}{Zeros of $\Phi_n(z;1)$}&\multicolumn{3}{|c|}{Zeros of $\Phi_n(z;1,a_n)$}\\
				\hline
				$n=5$ &$n=6$&$n=5$, $a_n=\frac{1}{n+1}-i$   &$n=6$, $a_n=-1.16$&$n=5, a_n=\frac{n}{n+1}$\\
				\hline
				0.294195-0.668367i&0.410684-0.639889i&0.303024-0.987019i&-1.07313&0.0\\
				\hline
				0.294195+0.668367i&0.410684+0.639889i&-0.113232-0.862069i&-0.941198&0.0\\
				\hline
				-0.375695-0.570175i&-0.205144-0.683797i&0.204197-0.629154i&0.348803-0.674876i&0.0\\
				\hline
				-0.375695+0.570175i&-0.205144+0.683797i&-0.617454-0.213182i&0.348803+0.674876i&0.0\\
				\hline
				-0.670332&-0.634112-0.287655i&-0.443202+0.4331161i&-0.350213-0.67387i&0.0\\
				\hline
				-&-0.634112+0.287655i&-&-0.350213+0.67387i&-\\
				\hline
		\end{tabular}}
		\captionof{table}{Zeros of  $\Phi_n(z;1)$ and $\Phi_n(z;1,a_n)$}
		\label{Zeros at gamma 1}
	\end{center}
\end{table}

\begin{minipage}{1.0\linewidth}
	\centering
	\begin{minipage}{0.4\linewidth}
		\begin{figure}[H]
			\includegraphics[width=\linewidth]{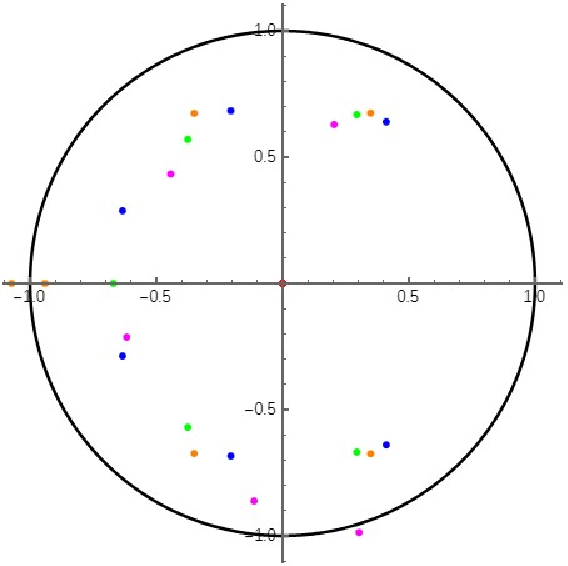}
			\captionof{figure}{{${}$}\\Zeros of $\Phi_5(z;1)$ (Green),\\ $\Phi_6(z;1)$ (blue),\\ $\Phi_5(z;1, -1.16)$ (orange),\\ $\Phi_5(z;1, \frac{n}{n+1})$ (red) and\\ $\Phi_5(z;1, \frac{1}{n+1}-i)$\\ (magenta).}
			\label{Location of zeros of QuasiOPUC for 1}
		\end{figure}
	\end{minipage}
	\hspace{0.06\linewidth}
	\begin{minipage}{0.4\linewidth}
		\begin{figure}[H]
			\includegraphics[width=\linewidth]{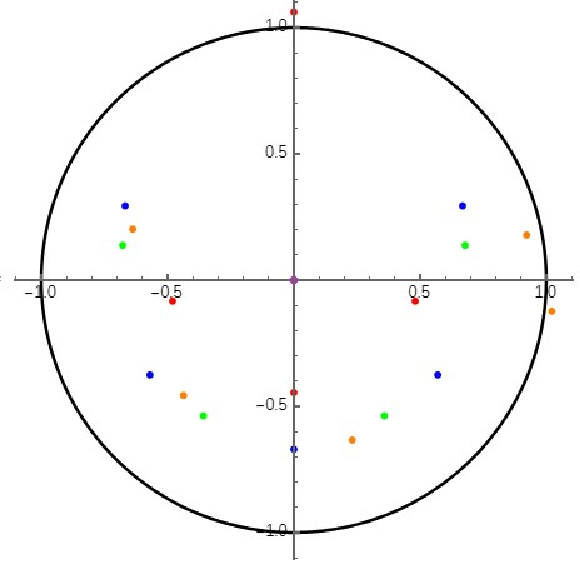}
			\captionof{figure}{{${}$}\\Zeros of $\Phi_4(z;i)$ (Green),\\ $\Phi_5(z;i)$ (blue),\\ $\Phi_4(z;i, \frac{n+1}{n}i)$ (red), \\ $\Phi_5(z;i, 1.1)$ (orange) and\\ $\Phi_4(z;1, \frac{n}{n+1}i)$ (magenta).}
			\label{Location of zeros of QuasiOPUC for i}
		\end{figure}
	\end{minipage}
\end{minipage}


%

We also observe the location of the zeros of the quasi-Christoffel polynomial at $\tilde{\gamma} = i$. In this case, the Verblunsky coefficients are given by $\alpha_n = \frac{(-1)^{n+2} i^{n+1}}{n+2}$, and the corresponding orthogonal polynomial is given by
\begin{align*}
	\Phi_{n-1}(z; i) = \frac{1}{z - i} \left( z^n - \frac{i^n}{n} \sum_{k=0}^{n-1} (-i)^k z^k \right).
\end{align*}


Thus, \eqref{QCP_Unit_circle} can be written as
\begin{align*}
	\Phi_{n}(z;i, a_n)=\frac{1}{z-i}\left(z^{n}(z-a_n)+\frac{a_n i^n}{n}\sum_{k=0}^{n-1}(-i)^kz^k-\frac{i^{n+1}}{n+1}\sum_{k=0}^{n}(-i)^kz^k\right).
\end{align*}

\begin{table}[htb!]
	\begin{center}
		\resizebox{!}{1.4cm}{\begin{tabular}{|c|c|c|c|c|}
				\hline
				\multicolumn{2}{|c|}{Zeros of $\Phi_n(z;i)$}&\multicolumn{3}{|c|}{Zeros of $\Phi_n(z;i,a_n)$}\\
				\hline
				$n=4$ &$n=5$&$n=4$, $a_n=\frac{n+1}{n}i$   &$n=5$, $a_n=1.1$&$n=4, a_n=\frac{n}{n+1}i$\\
				\hline
				0.67815+0.13783i&0.668367+0.294195i&1.0616i&1.02185-0.123134i&0.0\\
				\hline
				-0.67815+0.13783i&-0.668367+0.294195i&-0.480542-0.08298i&0.922406+0.178518i&0.0\\
				\hline
				-0.358285-0.53783i&-0.570175-0.375695i&0.480542-0.08298i&0.231037-0.633575i&0.0\\
				\hline
				0.358285-0.53783i&0.570175-0.375695i&-0.445624i&-0.638318+0.201751i&0.0\\
				\hline
				-&-0.670332i&-&-0.436978-0.456894i&-\\
				\hline
		\end{tabular}}
		\captionof{table}{Zeros of  $\Phi_n(z;i)$ and $\Phi_n(z;i,a_n)$}
		\label{Zeros at gamma i}
	\end{center}
\end{table}


Table \ref{Zeros at gamma i} and Figure \ref{Location of zeros of QuasiOPUC for i} shows that no zero of $\Phi_{n}(z;i)$ lie outside the unit disc. It is observed that for $a_n = \frac{n+1}{n}i$ and $a_n = 1.1$, at most one zero of $\Phi_n(z;i,a_n)$ lie outside the unit disc. However, when $a_n = \frac{n}{n+1}i$, the quasi-Christoffel polynomial $\Phi_n(z;i,a_n)$ becomes orthogonal with respect to the normalized Lebesgue measure.

These numerical experiments leads to the problem of finding the condition on the coefficients $a_n$:

\begin{enumerate}
\item	What are the conditions imposed on the coefficients $a_n$ in \eqref{QCP_Unit_circle} such that the polynomial $\Phi_{n}(z;\tilde{\gamma}, a_n)$
becomes orthogonal with respect to a certain measure?
\end{enumerate}

 \hspace{0.1cm}

	{\bf Acknowledgments}.  The second author acknowledges the support from Project No. NBHM/RP-1/2019 of National Board for Higher Mathematics (NBHM), DAE, Government of India.

\noindent
{\bf{\underline{Corresponding Author}}}.
A. Swaminathan, Department of Mathematics, IIT Roorkee, Roorkee, 247 667, India, Email: a.swaminathan@ma.iitr.ac.in

\noindent
{\bf{\underline{Ethical Approval}}}. No such approval is required as this work does not involve any such human and/or animal studies.

\noindent
{\bf{\underline{Data Availability}}}. Data sharing not applicable to this article as no datasets were generated or analysed during the current study.

\noindent
{\bf{\underline{Conflict of Interest}}}. There is no conflict of interest related to this manuscript.

\noindent
{\bf{\underline{Funding}}}. The second author acknowledges the support from Project No. NBHM/RP-1/2019 of National Board for Higher Mathematics (NBHM), DAE, Government of India.

\noindent
{\bf{\underline{Competing Interest}}}. The authors have no relevant financial or non-financial interests to disclose.

\noindent
{\bf{\underline{Authors contribution}}}. All authors contributed equally. All authors read and approved the final manuscript.

\end{document}